\documentclass[oneeqnum]{siamltex}
\usepackage{amssymb, amsmath, mathtools}
\mathtoolsset{showonlyrefs}

\usepackage[USenglish]{babel} 

\usepackage[inline, shortlabels]{enumitem}           
\setlist[enumerate,1]{label=(\roman*)}  
\usepackage{graphicx} 

\usepackage[round]{natbib}	
\usepackage[bookmarks=true, 
    pdfpagemode= UseNone, 
    backref= page, 
    pdftitle={stochastic dominance}, 
    pdfauthor={Vladimir Norkin, Alois Pichler}, 
    colorlinks=true, citecolor=blue, urlcolor=blue, linkcolor=blue, 
    pdfstartview= FitH]{hyperref} 
\usepackage{doi}
\usepackage[disable]{todonotes}

\usepackage{dsfont}
\usepackage{microtype}
\usepackage{orcidlink}

\newtheorem{remark}[theorem]{Remark}
\newtheorem{example}[theorem]{Example}
\usepackage{mathtools}

\DeclareMathOperator{\VaR}{{V@R}}	
\DeclareMathOperator{\AVaR}{{AV@R}}	
\DeclareMathOperator*{\argmin}{arg\,min}
\DeclareMathOperator{\E}{\mathds E}	        

\begin{document}

\pagestyle{empty} 


\title{\textbf{Portfolio reshaping under 1\textsuperscript{st} order stochastic dominance constraints
by the exact penalty function methods}}
\author{Vladimir Norkin\thanks{Gratefully acknowledges funding by Volkswagenstiftung (Volkswagen Foundation)} \and Alois Pichler\thanks{\orcidlink{0000-0001-8876-2429}~\href{https://orcid.org/0000-0001-8876-2429}{orcid.org/0000-0001-8876-2429}. DFG, German Research Foundation – Project-ID 416228727 – SFB 1410. Contact: \href{mailto:alois.pichler@math.tu-chemnitz.de}{alois.pichler@math.tu-chemnitz.de}}}
\date{December 20, 2022} 
\maketitle



\pagestyle{plain} 





\todo[inline]{Distinction to Ruszczynski: * Vladimirs method, * better portfolio with respect to some reference, * non-linear objective, * relax part of constraints and get higher return instead, * combination of constraints and objective}

\todo[inline]{journals: global optimization, Computational mgmt science, soft computing}

\todo[inline]{Structure and storyline:\\
* introduction: general non-linear optimization,\\ 
* relate to Vladimirs method\\
* example: portfolio optimization: bank fees $p(x^\top \xi)$ or $p(x)^\top \xi$ not reducible to integer constraints\\
* simple multi-period, fixed-mixed-strategy\\
* profit vs loss in the objective versus the constraints\\
* example from portfolio optimization\\
* risk-free asset as interior point with proof of convergence
}
\begin{abstract}
	The paper addresses general constrained and non-linear optimization problems.
	For some of these notoriously hard problems, there exists a reformulation as an unconstrained, global optimization problem.
	We illustrate the transformation, and the performance of the reformulation for a non-linear problem in stochastic optimization.
	The problem is adapted from portfolio optimization with first order stochastic dominance constraints.

	\medskip
	\textbf{Keywords.} Global optimization ·  exact penalty method · stochastic dominance
\end{abstract}

\section{Introduction}
We consider the general, constrained optimization problem
\begin{equation}\label{eq:0}
	\text{minimize }f(x),\qquad x\in \mathcal X,
\end{equation}
where the set $\mathcal X\subset \mathbb R^n$ constitutes the set of feasible solutions.
The projective exact penalty method (cf.\ \citet{Norkin_2022}) involves a map $\pi_{\mathcal X}$ such that 
\begin{align}\label{eq:Map}
	\pi_{\mathcal X}(\mathbb R^n)&\subset \mathcal X \text{ and }\\
	\pi_{\mathcal X}(x)&=x  \text{ for every } x\in \mathcal X.
\end{align}
With that, the global and local solutions of the constrained problem~\eqref{eq:0} and the unconstrained, global optimization problem
\begin{equation}\label{eq:1}
	\text{minimize }f\bigl(\pi_{\mathcal X}(x)\bigr)+ \|x-\pi_{\mathcal X}(x)\|, \qquad x\in \mathbb R^n,
\end{equation}
coincide.
Hence, the global, unconstrained optimization problem~\eqref{eq:1} can be considered instead of the constrained optimization problem~\eqref{eq:0}.

The map~$\pi_{\mathcal X}$ in~\eqref{eq:Map} is not specified.
For $\mathcal X$ convex, the projection
\begin{equation}\label{eq:7}
	\pi_{\mathcal X}(x)= \argmin\{\|y-x\|\colon y\in \mathcal X\}
\end{equation}
is well-defined and thus is a candidate for the global optimization problem~\eqref{eq:1}.
However, the projection~\eqref{eq:7} might not be available at cheap computational costs.

For a star domain $\mathcal X$, there exists a point~$x^0\in \mathcal X$ so that every line segment
$[x,x^0]\coloneqq \{\lambda\,x^0+(1-\lambda)x\colon \lambda\in [0,1] \}$ is fully contained in $\mathcal X$, provided that $x\in \mathcal X$.
The function
\begin{align}\label{eq:8}
	\pi_{\mathcal X}(x)&= \lambda_x\, x^0+(1-\lambda_x)x,\quad x\in\mathcal X,\\
	\shortintertext{where}
	\lambda_x&\coloneqq \sup\{\lambda\colon \lambda\, x^0+(1-\lambda) x\in\mathcal X\},
\end{align}
as well satisfies the conditions~\eqref{eq:Map}.
The problem~\eqref{eq:1} is not necessarily smooth nor convex.
So solving~\eqref{eq:1} requires applying non-smooth local and global optimization methods (cf.\ \citet{Norkin_2022}).
\medskip

The present paper outlines the described methodology for a financial portfolio optimization problem under specific constraints, namely 1\textsuperscript{st} order stochastic dominance constraints (FSD), where each feasible portfolio stochastically dominates a given reference portfolio. 
This means that a decision maker with non-decreasing utility function will prefer any feasible portfolio to the reference one.
Such portfolio optimization settings were considered in
\cite{Dentcheva_Ruszczynski_2003, Dentcheva_Ruszczynski_2006} (with the 2\textsuperscript{nd} order stochastic dominance constraints, SSD) and in 
\cite{Noyan_Rudolf_Ruszczynski_2006} (for linear problems with 1\textsuperscript{st} order stochastic dominance constraints).
Problems with FSD constraints are much harder than the ones with SSD constraints, because the former are non-convex.
\cite{Noyan_Ruszczynski_2008} reduce such problems to linear mixed-integer problems.
The present paper develops a different approach to such problems, which is applicable to the nonlinear case as well.
Our approach consists in an application of the exact penalty method to remove the FSD constraints and solution of the obtained penalty problem by non-smooth global optimization methods. 
\bigskip

\paragraph{\bf Outline of the paper}
First, Section~\ref{review} reviews the literature on decision-making and portfolio optimization under stochastic dominance constraints.

In Section~\ref{sec:Setting}, we set the problem of a portfolio optimization under 1\textsuperscript{st} order stochastic dominance constraints and provide some examples of such problems.
Next (cf.\ Section~\ref{sec:exactPF}), we reduce the portfolio optimization problem under 1\textsuperscript{st} order stochastic dominance constraints to the unconstrained problems by means of new exact projective non-smooth and discontinuous penalty functions. 

Forth, we review the successive smoothing method for local optimization of non-smooth and discontinuous functions.
This method is used as a local optimizer within the branch and bound framework for the global optimization of the penalty functions.
We finally give numerical illustrations (Section~\ref{sec:Numerical}) of the proposed approach to financial portfolio optimization under 1\textsuperscript{st} order stochastic dominance constraints on portfolios containing up to ten components with one risk-free asset.

\section{Literature review}\label{review}
The problem of financial portfolio optimization belongs to the class of decision-making problems under uncertainty.
The choice of a particular portfolio is accompanied by an uncertain result in the form of a distribution of future returns.
This and more general decision problems under stochastic uncertainty are studied in the theory of stochastic programming (cf.\ \cite{RuszczynskiShapiro2009}).
In the general case, the formalization of such problems is carried out using preferences defined on the set of possible uncertain results of decisions made.
Preferences establish partial order relationships on the set of decision outcomes, i.e., they satisfy the axioms of reflexivity, transitivity, and anti-symmetry.
Partial order relations make it possible to narrow the choice of preferable solutions to a subset of non-dominated alternatives.
Under additional assumptions about the properties of preferences, the latter can be represented in a numerical form, and then the problem of choosing preferred outcomes turns into a problem of multi-objective optimization.
Conversely, any numerical function on set of outcomes specifies a preference relation.
In essence, every possible scenario of decision outcomes can be considered as an optimization criterion.
A discussion of stochastic programming problems from this perspective can be found in \cite{GutjahrPi2013}.

In optimization and financial portfolio management problems, investments are often allocated according subject to utility and risk criteria.
The original settings of this type were proposed by \citet{Markowitz1952} and \citet{Roy_1952}, who used the mean return as a measure of utility and the variance of returns as a measure of risk.
An attractive feature of these formulations is the relative simplicity of the resulting optimization problems.

Subsequently, other utility and risk measures were proposed and used, such as quantiles, averaged quantiles, semi-deviations of returns from the average value, probabilities of returns falling into a profit or loss area, general coherent risk measures, and others 
(the references include \cite{Telser_1955, Kataoka_1963, Sen_1992, Prekopa_1995,  Kibzun_Kan_1996, Artzner_et_al_1999, Gaivoronski_Pflug_1999, Gaivoronski_Pflug_2005,  Benati_Rizzi_2007,  PflugRomisch2007, Luedtke_Ahmed_Nemhauser_2010,  Wozabal_Hochreiter_Pflug_2010,  Norkin_Boyko_2012, Kibzun_Naumov_Norkin_2013, Norkin_Kibzun_Naumov_2014, Kirilyuk_2015}, cf.\ also the references therein).
The corresponding decision selection problems are computationally more complex and require adapting known methods, or even developing new solution methods.
For example in \cite{RockafellarUryasev2000}, the problems of optimizing a financial portfolio in terms of averaged quantiles are reduced to a linear programming problem and can be effectively solved by existing software tools.
However, problems that involve quantiles or probabilities are much more difficult because they are non-convex and non-smooth with a possibly non-convex and disconnected admissible region.
The problem may be even harder, if the return depends non-linearly on the portfolio structure (see, e.g., the discussion of the properties of these problems in \cite{Gaivoronski_Pflug_1999, Gaivoronski_Pflug_2005}).
For example, in the works by \cite{Noyan_Rudolf_Ruszczynski_2006, Benati_Rizzi_2007, Luedtke_Ahmed_Nemhauser_2010, Norkin_Boyko_2012, Kibzun_Naumov_Norkin_2013, Norkin_Kibzun_Naumov_2014}, such problems are reduced to problems of mixed-integer programming in the case of a discrete distribution of random data.
\cite{Gaivoronski_Pflug_2005} developed a special method for smoothing a variational series to optimize a portfolio by a quantile criterion.
\cite{Wozabal_Hochreiter_Pflug_2010} give a review of the quantile constrained portfolio selection problem, present a difference-of-convex representation of involved quantiles, and develop a branch and bound algorithm to solve the reformulated problem.

On the other hand, natural relations of stochastic dominance of the first, second and higher orders are known on the set of probability distributions.
For example, the first order stochastic dominance relation is defined as the excess of one distribution function over another distribution function.
The relation of stochastic dominance of the second order is determined by the relation of the integrals of the distribution functions of random variables.
With the second-order stochastic dominance relation, the decision maker's negative attitude towards risk can be expressed (see a discussion of these issues in \cite{Muller_Stoyan_2002, Dentcheva_Ruszczynski_2013}).

A natural question arises about the connection between decision-making problems in a multi-criteria formulation and in terms of certain preference relations, in particular, stochastic dominance relations.
The connection between mean-risk models and second-order stochastic dominance relations was studied in the works by \cite{Rusz1999, Ruszczynski2001, RuszOgryczak}.
In the works by \cite{Dentcheva_Ruszczynski_2012, Dentcheva_Ruszczynski_2013} it is shown under what conditions the problem of decision-making in terms of preference relations is reduced to the problem of optimizing a numerical indicator. 

\cite{Dentcheva_Ruszczynski_2003, Dentcheva_Ruszczynski_2006} proposed a mixed financial portfolio optimization model in which a numerical criterion is optimized, and constraints are specified using second-order stochastic dominance relations.
The feasible set in this setting consists of decisions, which dominate some reference one and are preferred by any risk averse decision maker.
In the case of a discrete distribution of random data, the problem is reduced to a linear programming problem of (large) dimension.
These works have given rise to a large stream of work on stochastic optimization problems under second-order stochastic dominance constraints (see the reviews by 
\cite{Dentcheva_Ruszczynski_2011,  FabianMitra, GutjahrPi2013}).

\cite{Noyan_Rudolf_Ruszczynski_2006, Noyan_Ruszczynski_2008, Dentcheva2014} considered similar mixed problems, but with first-order stochastic dominance relations in the constraints.
To solve them, a method of reduction to problems of linear mixed-integer programming with subsequent continuous relaxation of Boolean constraints and the introduction of additional cutting constraints is proposed. 
\cite{Dentcheva_Henrion_Ruszczynski_2007} studied the stability of these problems with respect to perturbation of the involved distributions on the basis of general studies of the stability of stochastic programming problems.
\medskip

In this article, we consider similar financial portfolio optimization problems under 1\textsuperscript{st} order stochastic dominance constraints, but from a different point of view and apply a different solution approach that is also applicable to problems with nonlinear random return functions.
This problem is viewed as the problem of optimizing the risk profile of the portfolio according to the preferences of the decision maker.
Namely, the decision maker sets the desired risk profile (the form of the cumulative distribution function) and tries to find an acceptable portfolio that dominates this risk profile.
This is one statement, and the other is that, under the condition of the existence of an admissible portfolio, i.e., portfolio dominating some reference portfolio, it can be any index or risk-free portfolio, choose a portfolio with the desired risk profile by optimizing one or another function (for example risk measures, as average quantiles, etc.).
In this way it is possible to satisfy the needs of both a risk-prone decision maker and a risk-averse decision maker.
In the first case, the average quantile function for high returns is maximized, and in the second case, the mean quantile function for low returns is maximized, in both cases with a lower bound on the quantile risk profile.
In the first case, the risk profile is stretched, and in the second case, it is compressed and becomes more like a profile of a deterministic value.
The reshaping of the risk profile can be made both through selection of different objective function and by adding new securities to the portfolio, e.g., as commodities, etc., cf.\ \cite{PichlerWestgaard}.

In the problems under consideration, the objective function is non-linear, non-convex, and possibly non-smooth or even discontinuous, and the number of constraints is continual (uncountable).
But when the reference profile has a stepped character, one can limit oneself to a finite number of restrictions by the number of steps of the reference profile.
In this problem, the admissible area may turn out to be non-convex and disconnected. 
Thus, the problem under consideration is a global optimization problem with highly complex and nonlinear constraints.
To solve it, we first reduce it to an unconstrained global optimization problem by applying new penalty functions, namely, discontinuous penalty functions as in  \cite{Batukhtin_1993, Knopov_Norkin_2022} and the so-called projective penalty functions as in \cite{Norkin_2020, Norkin_2022}.
In the first case, the objective function outside the allowable area is extended by large but finite penalty values, and in the second case, it is extended at infeasible points by summing the value of the objective function in the projection of this point onto the feasible set and the distance to this projection.
In this case, the projection is made in the direction of some known internal feasible point.
After such a transformation, the problem is still a complex unconstrained global optimization problem.
In order to solve it, we further apply the method of successive smoothing of this penalty function, i.e., we minimize successive smoothed approximations of the penalized function, starting from relatively large smoothing parameters with its gradual decreasing to zero.
It is known that the smoothed functions can be optimized by the method of stochastic gradients, where the latter have the form of finite difference vectors in random directions,
cf.\ \cite{Mayne_Polak_1984, Mikhalevich_Gupal_Norkin_1987,  Ermoliev_Norkin_Wets_1995,  
Nesterov_Spokoiny_2017}. 
Here, smoothing plays a dual role. Firstly, it allows optimizing non-smooth and discontinuous functions and, secondly, it levels out shallow local extrema.
Although smoothing makes it possible to ignore small local extrema, it does not guarantee convergence to the global extrema.
Therefore, we put the method of sequential smoothing in a general scheme of the branch and bound method, where the smoothing method plays the role of a local optimizer on subsets of the optimization area.
The scheme of the branch and bound method is designed in such a way that the calculations are concentrated in the most promising areas of the search for the global extrema.

This article describes the financial portfolio optimization model with 1\textsuperscript{st} order stochastic dominance constraints and illustrates the proposed approach to its solution on the problems of reshaping the risk profile of portfolios of small dimension.
At the same time, the results of changing the shape of the risk profile are presented in graphical form, which allows visually comparing the resulting profile with the reference one, and, if necessary, continue adaptation of the profile to the preferences of the decision maker.

\section{Mathematical problem setting}\label{sec:Setting}
The financial portfolio is described by a vector $x= (x_1,\dots,x_n)'$ of values $x_i$ and by a random vector of returns $\omega= \left(\omega_1,\dots,\omega_n \right)'\in \Omega$ of assets, $i= 1,\dots,n$, in some fixed time interval; $(\cdot)'$ means transposition of a vector.
Denote by
\[	X= \left\{x\in \mathbb R^n\colon \sum\nolimits_{i=1}^n x_i\le 1,\ x_i\ge c_i\ge -\infty\right\}\]
the set of admissible portfolios with a unit maximal total cost of the whole portfolio, $c_i$~is a lower bound on the value of component $i$ of the portfolio (e.g., a short-selling constraint or a limitation on borrowing assets).
In the definition of the set~$X$, the inequality $\sum\nolimits_{i=1}^n x_i\le 1$ is used, which means that  $x_0= 1- \sum\nolimits_{i=1}^nx_i$~-- the non-invested funds~-- have zero yield.
The portfolio is characterized by a random return $f(x,\omega )=\omega' x$, by the mean return $\mu (x)= \E_\omega f(x,\omega)= \sum_{i=1}^n x_i \E_\omega\omega_i$ for the considered period of time and by the variance of return $\sigma^2(x)=\E_\omega\bigl( f(x,\omega)-\E_\omega f(x,\omega)\bigr)^2$, where $\E_\omega$ denotes the mathematical expectation with respect to the distribution of random variable $\omega$, and $(\cdot)'$ denotes the transposition of a vector.

According to Markowitz (cf.\ \cite{Markowitz1952, Markowitz_1959}), the portfolio is optimized by two criteria, the mean return $\mu(x)$ and the standard deviation. 
A set of non-dominated portfolios
$\Gamma= \left\{\left(y, \sigma^*(y)\right)\colon y\in \mathbb R\right\}$ with $\sigma^{*2}(y)= \min_{x\in X,\,\mu(x)\ge y}\sigma^2(x)$, is called the \emph{efficient frontier} (boundary).
An optimal portfolio is selected from the efficient frontier by optimization of some utility function $\Phi(\mu ,\sigma)$, defined in the “risk--return” plane~$(\sigma,\mu)$.

The classical financial portfolio models assume a linear dependence of the return on the portfolio structure, for which alternative problem reformulations are available.
Non-linearities appear when random returns are modeled by some parametric distribution.
The following example provides one more example of such a nonlinear dependence.

\begin{example}[Nonlinear return in a dynamic portfolio model]
	Consider a portfolio of $n$ assets through discrete time intervals $t= 1,\dots,T$.
	In time period~$t$, the value of assets in category~$i$ grows by a factor $\omega_i^t$, where $\omega^t$, $t= 1,2,\dots,T$, is a sequence of $n$-dimensional random variables.
	Denote $X= \{x\in \mathbb R^n\colon x_i\ge 0, i=1,\dots,n,\ \sum_{i=1}^n x_i =1\}$ and $x^t\in X$ the structure of the portfolio at time interval $t$ (the case of the fixed mix portfolio $x= x^t$ for all~$t$ was considered in \cite{Norkin_Pflug_Ruszczynski_1998}).
	After each time period one rebalances the portfolio to have the proportions~$x^t$ of the values of assets in various categories. 
	Each selling/\,buying of asset induces transaction costs of a fraction $\alpha_i$ of the amount traded.
	The problem is to find the strategy $x= (x^1,\dots, x^n)$ that maximizes, e.g., the expected portfolio wealth after~$T$ periods.
	Denote the wealth at the beginning of period~$t$ by $W(t)$ and assume that $W(1)=1$.
	Then, at the end of period~$t,$ the wealth in category~$i$ equals
	$x_i^t W(t)\omega_i^t$, while the transaction costs necessary to establish the proportion $x_i^t$ are equal to
	$\alpha_i\cdot\Big|x_i^t\, W(t)\omega_i^t- x_i^t \sum_{j=1}^n x_j^t\,  W(t)\omega_j^t\Big|$ and thus
	\[	W(t+1)= W(t) \sum_{i=1}^n \Bigl(x_i^t\, \omega_i^t-\alpha_i\Big|x_i^t\, \omega_i^t -x_k^t \sum_{j=1}^n x_j^t\, \omega_j^t\Big|\Bigr).\]
	The random return therefore has the form
	\[	f(x,\omega)= \prod_{t=1}^T\Bigl\{\sum_{i=1}^n\Bigl( x_i^t\omega_i^t-\alpha_i\Big|x_i^t\omega_i^t -x_i^t\sum_{j=1}^n x_j^t \omega_j^t\Big|\Bigr)  \Bigr\}, \quad x\in X^n.\]
\end{example}

	Suppose the random vector~$\omega$ is given by a discrete (an empirical, e.g.)\ distribution $\{\omega^1,\ldots,\omega^m\}$ with equiprobable values $\omega^i \in \mathbb R^n$, $i=1,2,\ldots,m$. 
	Then the average portfolio return 
	$\mu(x)=\E\omega' x$ and the \emph{cumulative distribution function} (CDF) of the portfolio return
	$\mathcal F_x(t)=Pr\{\omega' x\le t\}$ are given by
	\begin{align}
		\mu_m(x)&=\frac1m\sum_{i=1}^m (\omega^i)' x\\
	\shortintertext{and}
		\mathcal F_{x,m}(t)&=m^{-1}\{\#i: (\omega^i)' x < t\}.
	\end{align}

	The portfolio optimization problem under 1\textsuperscript{st} order stochastic dominance constraint has the form
	\begin{align}\label{obj1}
		\text{maximize }&f_m(x)\\
		\text{ subject to }&\mathcal F_{x,m}(t)\le \mathcal F_\mathit{ref}(t)\text{ for all }t\in\mathbb R \text{ and }x\in X,	\label{sdc1}
	\end{align}
where $\mathcal F_\mathit{ref}(t)$ is a some reference cumulative distribution function 
(continuous from the left), for example
$\mathcal F_{x_\mathit{ref},\,m}(t+\delta(t))$, $\delta(t)\ge 0$, which is also called a \emph{reference risk profile}.

Here, the function $\mathcal F_x(\cdot)$ is continuous from the left and $\delta(\cdot)$ and $\mathcal F_\mathit{ref}(\cdot)$ are assumed to be continuous from the left.
In such a case, the function $\mathcal F_{x,m}(t)$ appears to be lower semicontinuous in $(x,t)$, hence the sets 
$\{x\in\mathbb R^n\colon \mathcal F_{x,m}(t)\le \mathcal F_\mathit{ref}(t)\}$
are closed and the sets $\{x\in X\colon \mathcal F_{x,m}(t)\le \mathcal F_\mathit{ref}(t)\}$ are compact for each $t$.

If the reference function $\mathcal F_\mathit{ref}(t)$ has only finitely many jumps at
\[	\mathcal T_\mathit{ref}= \left\{t_1,t_2,\ldots,t_k\right\},\] then 
\[	\left\{x\in X\colon \mathcal F_{x,m}(t)\le \mathcal F_\mathit{ref}(t)\text{ for all } t\in \mathbb R\right\}
	= \left\{x\in X\colon \mathcal F_{x,m}(t)\le \mathcal F_\mathit{ref}(t), t\in \mathcal T_\mathit{ref} \right\} \]
and the feasible set in~\eqref{sdc1} is compact.
\bigskip

As objective $f_m(x)$ in the master problem~\eqref{obj1} we consider
\begin{itemize}
	\item the mean value $\mu_m(x)$,
	\item some Value-at-Risk function $\VaR_\gamma(x)$, $\gamma\in (0,1)$,
	\item the average Value-at-Risk function 
		\begin{equation}\label{avars}
			\AVaR_{\alpha,\beta}(x)\coloneqq\frac 1{\beta-\alpha}\int_\alpha^\beta \VaR_{\gamma}(x)d\gamma,\quad 0\le \alpha< \beta\le 1,
		\end{equation}
\end{itemize}
in particular, $\AVaR_\gamma(x)=\AVaR_{\gamma,1}(x)$ and $\AVaR_{0,1}(x)= \mu_m(x)$.
\medskip

We formally can associate some random variable $\xi_\mathit{ref}$ with the CDF $\mathcal F_\mathit{ref}(t)$.
Then the family of inequalities~\eqref{sdc1} states that random variable $\xi_x=\omega'x$
1-st order stochastically dominates the random variable $\xi_\mathit{ref}$.
We further remark there there can be several stochastic dominance constraints with corresponding reference CDF $\mathcal F_\mathit{ref}^i$, which can be replaced by the single CDF $\mathcal F_\mathit{ref}= \min_i\mathcal F_\mathit{ref}^i$.
\medskip

\paragraph{\bf Alternative problem formulation}
A reformulation of the problem~\eqref{obj1}--\eqref{sdc1} involves the inverse of the cumulative distribution functions instead of the CDF.
To this end let 
\[	\mathcal Q_{x,m}(\alpha)\coloneqq \sup_{t\in\mathbb R}\left\{t\colon \mathcal F_{x,m}(t)\le\alpha\right\},
\quad \alpha\in[0,1],\]
be the return quantile (generalized inverse) function associated with the decision $x$, and $\mathcal Q_\mathit{ref}(\alpha)$ be some reference quantile function, continuous from above.%
	\footnote{Note, that $\mathcal F_{x,m}(\cdot)$ is the upper quantile function, i.e., it is continuous from the right (upper semicontinuous), the function $\mathcal Q_\mathit{ref}(\cdot)$ is assumed to be continuous from the right.}

Consider the problem
\begin{align}
	\text{maximize }&f_m(x)\label{obj2}\\
	\text{subject to }& \mathcal Q_\mathit{ref}(\alpha)\le \mathcal Q_{x,m}(\alpha)\ \text{ for all }\alpha\in[0,1]\text{ and } x\in X,\label{sdc2}
\end{align}


The quantile function
	$\mathcal Q_{x,m}(\alpha)\coloneqq\sup_{t\in\mathbb R}\left\{t\colon \mathcal F_{x,m}(t)\le \alpha \right\}$
is a marginal function, so under the assumptions made it is upper semicontinuous in $(x,\alpha)$ by \citet[Ch.~1, Sec.~1, Prop.~21]{Aubin1984}.
If the reference function $\mathcal Q_\mathit{ref}$ has a step like character with steps at ${\cal A}_{ref}= \{0=\alpha_1,\alpha_2,\dots,\alpha_k=1\}$,
then the feasible set
\begin{align}
	\left\{x\in X\colon \mathcal Q_\mathit{ref}(\alpha)\le \mathcal Q_{x,m}(\alpha) \text{ for all} \alpha\in [0,1]\right\}\\
	=\left\{x\in X\colon \mathcal Q_\mathit{ref}(\alpha)\le \mathcal Q_{x,m}(\alpha)\; \forall \alpha\in {\cal A}_\mathit{ref}\right\}
\end{align}
is compact.

Employing different objective functions $f_m(x)$ allows to reshape the risk profile 
$\mathcal F_x(t)$ and $\mathcal Q_x(\alpha)$ in a desirable manner.
For example, the problem 
\begin{align}
	\text{maximize }&\AVaR_{\gamma,1}(x)\\
	\text{subject to }&
		\mathcal Q_{x_\mathit{ref}}(\alpha)-\delta(\alpha)\le  \mathcal Q_{x,m}(\alpha),\  \delta(\alpha)\ge 0\ (\alpha\in[0,1]) \text{ and }x\in X,
\end{align}
%
can be used for searching more risky but potentially more profitable portfolios than some 
reference one $x_\mathit{ref}$ with risk profile $\mathcal Q_{x_\mathit{ref}}(\alpha)$ and step back function $\delta(\alpha)$.
The problem
\begin{align}
	\text{maximize }& \AVaR_{0,\gamma}(x)\\	
	\text{subject to }&
		\mathcal Q_{x_\mathit{ref}}(\alpha)-\delta(\alpha)\le \mathcal Q_{x,m}(\alpha),\  \delta(\alpha)\ge 0\ (\alpha\in [0,1])\text{ and }x\in X
\end{align}
%
can be used to obtain less risky and less profitable portfolio than a reference portfolio based on the allocation $x^0$.
Note, however, that the objective functions in these problems can be discontinuous.
\medskip

\begin{example}[Portfolio selection under a single Value-at-Risk ($\VaR$) constraint, cf.\ \cite{Wozabal_Hochreiter_Pflug_2010} and corresponding references therein]
	Let $\mathcal F_x(t)$ be the CDF and $Q_x(\alpha)$ be the 
	$\alpha$-quantile of the random return $\xi_x=\omega'x$ for some fixed $\alpha$, 
	$q_\alpha$ be the reference value for $Q_x(\alpha)$, 
	and $\tau\le\min_{1\le i\le m}\omega_i$ a.s. Consider the problem
	\begin{align}
		\text{maximize   }& f_m(x)\\
		\text{subject to }& Q_x(\alpha) \ge q_\alpha,\label{quant_ineq}
	\end{align}
	where $\alpha$ is a fixed risk level.
	With the reference CDF
	\[	\mathcal F_{\alpha,\tau}(t)\coloneqq
		\begin{cases}
			0 & \text{if }t\le \tau,\\
			\alpha & \text{if }\tau< t\le q_\alpha,\\
			1 & \text{if }q_\alpha < t,
		\end{cases}\]
	the inequality~\eqref{quant_ineq} is equivalent to the constraints 
	$\mathcal F_x(t)\le \mathcal F_{\alpha,\tau}(t)$ for all $t$.
	In this way, multiple quantile constraints can be reduced to a stochastic dominance constraint.
\end{example}
\medskip

\begin{example}[Decision making under catastrophic risks, \cite{Norkin_2006}]
	Catastrophic risks, as catastrophic floods, earthquakes, tsunami, etc., designate some “low probability~-- high consequences” events.
	Usually, they are described by a list of possible extreme events (indexed by $i= 1,2,\dots, I$) that can happen once in ten, fifty, hundred, etc.\ years. 
	Decision-making under catastrophic risks means designing a certain mitigation measures to prevent unacceptable losses.
	\cite{Norkin_2006} proposed the following framework for decision-making under catastrophic risks.

	Let $x\in X$ denote a decision (a complex of countermeasures) from some set $X$ of possible decisions, each associated with costs $c(x)$.
	For each kind of event $i$, experts can define reasonable (“acceptable”) levels of losses $q_i$ due to this event, $q_1<\ldots < q_I$.
	Suppose we can model each event $i$, its consequences and losses $l_i(x)$ under the decision $x\in X$.
	Then the corresponding decision-making problem is
	\begin{align}
		\text{minimize }& c(x) \label{cat1}\\
		\text{subject to }& l_i(x)\le q_i,\quad i=1,2,\dots, I. \label{cat2}
	\end{align}
	Although the framework does not include explicit probabilities of the events $i$, we can formally introduce probabilities $p_1>p_2>\ldots>p_I$, e.g., $p_1=1/10$, $p_2=1/50$, $p_3=1/100$, etc., that event~$i$ happens in any given year.
	Define also the absolute losses $q_\infty>q_i$ for all $i$.
	By defining two CDF,
	\[	\mathcal F_\mathit{ref}(t)\coloneqq
		\begin{cases}
			0,& \text{if }t<0,\\
			p_1,& \text{if }0\le t<q_1,\\
			\sum_{k=1}^i p_k,& \text{if }q_{i-1}\le t< q_i,\ 2\le i\leq I,\\
			1,& \text{if }q_I\le t
		\end{cases} \]
	and
	\[	\mathcal F_x(t)\coloneqq
		\begin{cases}
			0,& \text{if }t< 0,\\
			p_1,& \text{if }0\le t< l_1(x),\\
			\sum_{k=1}^i p_k,&\text{if } l_{i-1}(x)\le t< l_{i}(x), 2\le i\le I,\\
			1,&  \text{if }l_I(x)\le t,
		\end{cases} \]
	we can express the constraints~\eqref{cat2} as
	$\mathcal F_x(t)\ge \mathcal F_\mathit{ref}(t)$ for all $t$, i.e., in terms of 1\textsuperscript{st} order stochastic dominance.
\end{example}




\section{The solution approach: exact penalty functions}\label{sec:exactPF}
In case of a discrete random variable $\omega$, the function $\mathcal F_{x,m}(t)$ in~\eqref{sdc1} is discontinuous in $t$ and $x$.
For the solution of the problem~\eqref{obj1}--\eqref{sdc1}, we apply the exact discontinuous and the exact projective penalty functions from \cite{Batukhtin_1993, Norkin_2020, Norkin_2022, Knopov_Norkin_2022,  Galvan_et_al_2021}.

\subsection{Finding a feasible solution}
If the reference function $\mathcal F_\mathit{ref}(\cdot)$ has a stepwise character with jumps at step points
$\mathcal T_\mathit{ref}= \{t_1,\dots,t_k\}$, we may set
\[	G_m(x)\coloneqq
		\max_{t\in\mathcal T_\mathit{ref}}\big(\mathcal F_{x,m}(t)- \mathcal F_\mathit{ref}(t)\big). \]
With that, the constraints~\eqref{sdc1} are equivalent to the inequality $G_m(x)\leq 0$.

To find a feasible solution $x^0$ for the problem~\eqref{obj1}--\eqref{sdc1}, we solve the problem
\begin{equation}\label{funcG}
	\min_{x\in X}\, G_m(x).
\end{equation}
If for some $x^0\in X$ it holds that $G_m(x^0)\le 0$, then $x^0$ is a feasible solution of the problem \eqref{obj1}--\eqref{sdc1}.

Similarly, to find a feasible solution of problem \eqref{obj2}--\eqref{sdc2}, we solve the problem
\begin{equation}\label{funcH}
	\min_{x\in X}\ H_m(x)\coloneqq \max_{\alpha\in{\cal A}_\mathit{ref}}\left\{\mathcal Q_\mathit{ref}(\alpha)-\mathcal Q_{x,m}(\alpha)\right\},
\end{equation}
where $\mathcal A_\mathit{ref}$ is the set of jump points of $\mathcal Q_{x,m}(\cdot)$.

\subsection{Structure of the feasible set}
The structure of the feasible set of the problems~\eqref{obj1}--\eqref{sdc1} and \eqref{obj2}--\eqref{sdc2} heavily depends on the choice of the reference profiles
$\mathcal F_\mathit{ref}(t)$ and $\mathcal Q_\mathit{ref}(\alpha)$.
Figure~\ref{figfeaset} illustrates disconnected and non-convex feasible sets for a portfolio consisting of 3 components only.

Suppose the portfolio $x^{0}= (0.31,0.69,0)'$ includes the first two assets of Table~\ref{table2} from the appendix, with random return~$\omega$ given by the first two columns of this table.
Let $\mathcal F_0(t)=\Pr\{\omega x^{0}<t\}$ be the risk profile of this portfolio and
$\mathcal F_\mathit{ref}(t)=\mathcal F_0(t+\delta)$, $\delta=0.05$, be the reference risk profile.
The left figure in Fig.~\ref{figfeaset} displays the shape of the non-convex disjoint feasible set of problem~\eqref{obj1}--\eqref{sdc1} for this example.
The figure in the middle gives an example of the feasible set of problem~\eqref{obj1}--\eqref{sdc1}, when a risk-free asset is infeasible.
The right figure of Fig.~\ref{figfeaset} corresponds to a case of feasible risk-free portfolio.

A possibly complex structure of the feasible set motivates us to consider discontinuous penalty functions.
Besides, in case of a discrete distribution of the portfolio return, the effective objective functions~\eqref{avars} also can be discontinuous.

\begin{figure*}[ht]       
  \includegraphics[width=0.32\textwidth]{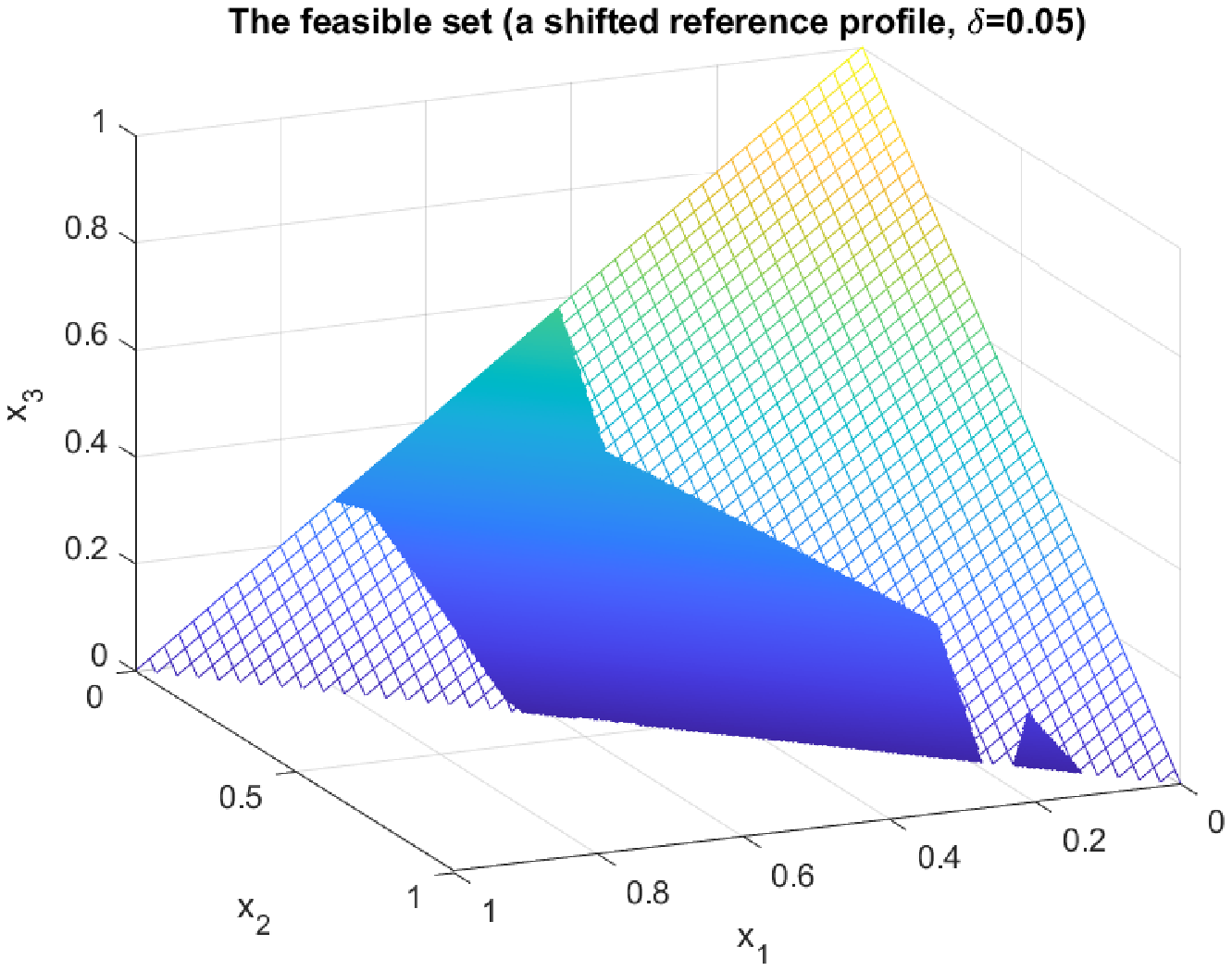}
	\includegraphics[width=0.32\textwidth]{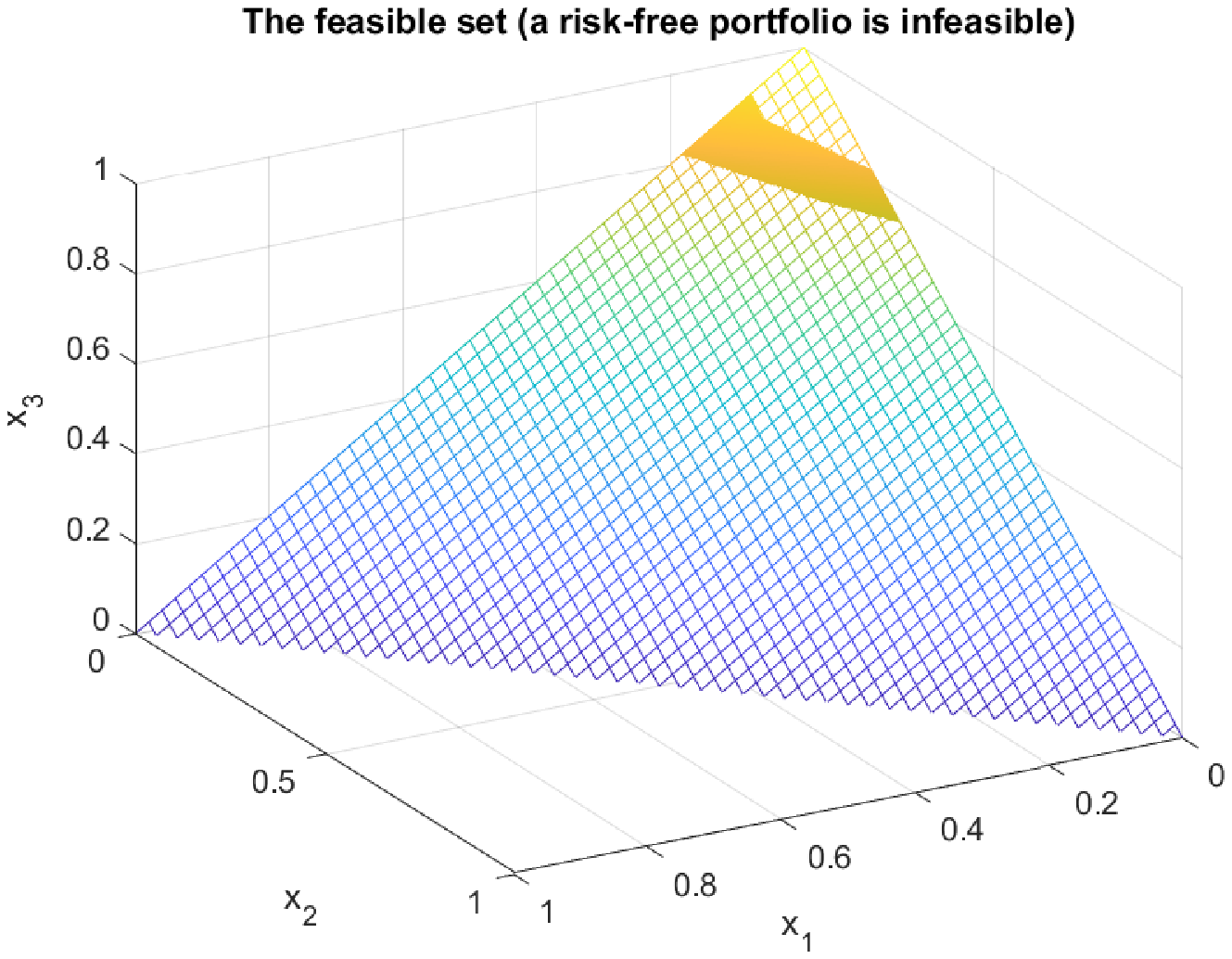}
	\includegraphics[width=0.32\textwidth]{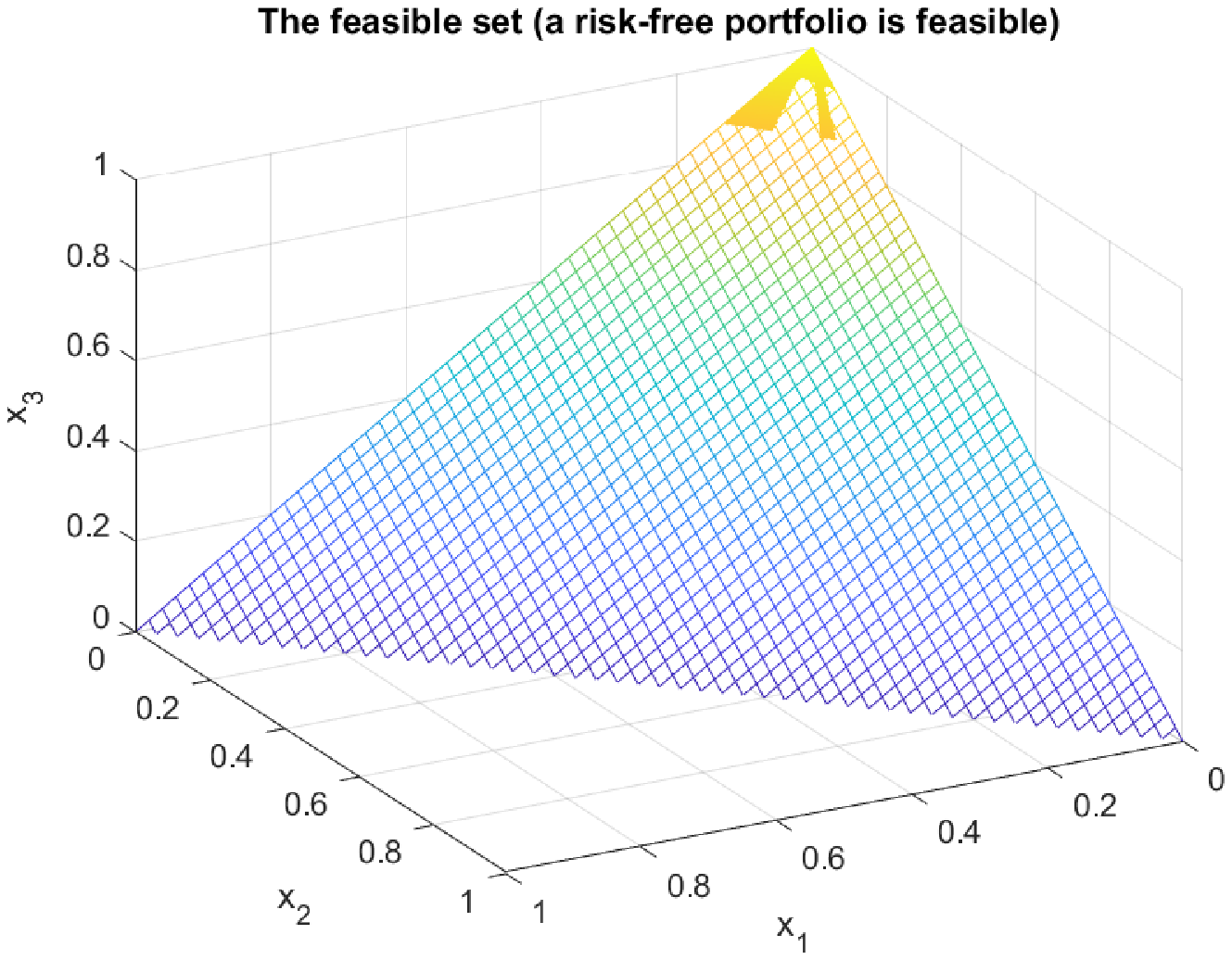}
	\caption{Possible (disconnected and non-convex) shapes of feasible sets under 1\textsuperscript{st} order SDC
	($x_\mathit{ref}= (x_1=0.31,x_2=0.69)$, $\delta=0.05, 0, 0$)
} \label{figfeaset}
\end{figure*}

\subsection{Exact discontinuous penalty functions}
\label{edpf}
To find an optimal solution for problem~\eqref{obj1}--\eqref{sdc1}, we solve the problem
\begin{equation}\label{funcF}
	\max_{x\in X}F_m(x)\coloneqq f_m(x)-c-\max\{0,G_m(x)\},
\end{equation}
where $c\le \max_{x\in X}f_m(x)$, for example, $c=f_m(x^0)$ for some feasible $x^0$.
Obviously, the global maximums of the problems~\eqref{obj1}--\eqref{sdc1} and \eqref{funcF} coincide.

We can further remove the constraint $x\in X$ from problem~\eqref{funcF} by subtracting the exact projective penalty term $\|x-\pi_X(x)\|$, where $\pi_X(x)$ is the projection of~$x$ on~$X$, from the objective function~\eqref{funcF}, 
\begin{equation}\label{funcF2}
	\max_{x\in\mathbb R^n}F_m(x)\coloneqq f_m(\pi_X(x))-c-\max\{0,G_m(\pi_X(x))\}-\|x-\pi_X(x)\|.
\end{equation}

For the problem~\eqref{obj2}--\eqref{sdc2}, the corresponding penalized problem has the form:
\begin{equation}\label{funcPhi}
	\max_{x\in \mathbb R^n} \Phi_m(x)\coloneqq f_m(x)-c-\max\{0,H_m(x)\}-\|x-\pi_X(x)\|.
\end{equation}

\subsection{Exact projective penalty functions}
\label{eppf}
Let $x^0\in X$ be some feasible solution of problem~\eqref{obj1}--\eqref{sdc1}, i.e., $x^0\in X$ and $G_m(x^0)\le 0$.
For any $x\in \mathbb R^n$ denote $x_\lambda= (1-\lambda)x^0+ \lambda\, x$. Define a projection point
\begin{equation}\label{noneuclidproj}
	p_{G_m}(x)=
	\begin{cases}
		x &     \text{if }G_m(x)\le 0,\\
		x_{\lambda_x}& \text{if }G_m(x)>0,
	\end{cases}
\end{equation}
where
\[	\lambda_x=\sup\{\lambda\in [0,1]\colon G_m(x_\lambda)\le 0\}.\]
Now instead of the constrained problem~\eqref{obj1}--\eqref{sdc1} consider the unconstrained problem
\begin{equation}\label{doubleprojection1}
	\max_{x\in\mathbb R^n}F_m(x)\coloneqq f_m\left(p_{G_m}(\pi_X(x))\right)-\|p_{G_m}(\pi_X(x))-\pi_X(x)\|-\|x-\pi_X(x)\|.
\end{equation}

Similarly, instead of constrained problem~\eqref{obj2}--\eqref{sdc2}, we can consider the unconstrained problem
\begin{equation}\label{doubleprojection2}
	\max_{x\in\mathbb R^n}F_m(x)\coloneqq f_m\bigl(p_{H_m}(\pi_X(x))\bigr)-\|p_{H_m}(\pi_X(x))-\pi_X(x)\|-\|x-\pi_X(x)\|.
\end{equation}
Following \citet[Theorem~4.4]{Norkin_2022}, the global maximums of problems~\eqref{obj1}--\eqref{sdc1} and~\eqref{doubleprojection1} coincide, and, if the mapping $p_{G_m}(\pi_X(x))$ is continuous, also local maximums of the both problems coincide, i.e., the optimization problems are equivalent.
The mapping $p_{G_m}(\pi_X(x))$ is continuous, if the feasible set is convex and the projection mapping~\eqref{noneuclidproj} uses an internal point $x^0$ of the feasible set of the considered problem.
For example, a feasible risk-free portfolio can represent such point.
\medskip

\begin{remark}[Computational aspects]
	The calculations of the projections $p_{G_m}(\cdot)$ and $p_{H_m}(\cdot)$ requires finding roots of equations
	$\phi(\lambda)\coloneqq G_m(x_\lambda)=0$ and 
	$\psi(\lambda)\coloneqq H_m(x_\lambda)=0$.
	This requires multiple evaluations of the functions $\phi(\lambda)$ and $\psi(\lambda)$, and hence, multiple construction of  
	$\mathcal F_{x_\lambda,m}(t)=\Pr\{\omega'x_\lambda\le t\}$ or
	$\mathcal Q_{x_\lambda,m}(\alpha)$ for different portfolios~$x_\lambda$.
	This may take considerable time in case of large number of observations~$m$.
	However, in case of a specific feasible point, namely a risk-free (feasible) portfolio, these functions can be easily found through 
	$\mathcal F_{x,m}(t)$ and $\mathcal Q_{x,m}(\alpha)$ as the following statements show.
\end{remark}
\medskip

\begin{proposition}
	Let $x^0= (1,0,\dots,0)$ be a risk-free portfolio with fixed return~$r$, $x\in X$ be an arbitrary portfolio, with the random return $f(x,\omega)=\omega'x$, the return cumulative distribution function $\mathcal F_x(t)$, $t\in\mathbb R$, and the corresponding inverse (quantile) one $\mathcal Q_x(\alpha)$, $\alpha\in [0,1]$,
	$X=\{x\in \mathbb R^n\colon \sum_{i=1}^n x_i=1,\ x_i\ge 0,\ i=1,\dots,n\}$.
	Let us consider a mix portfolio of the form $x_\lambda= \lambda \, x+(1-\lambda)\, x^0$, $\lambda\in [0,1]$.
	Its distribution function and inverse distribution function are expressed through $\mathcal F_x(t)$ and $\mathcal Q_x(\alpha)$ as
	\begin{align}
		\mathcal F_{x_\lambda}(t)&=\mathcal F_{x}\bigl((t-(1-\lambda)r)/\lambda\bigr), \\
		\shortintertext{and}
		\mathcal Q_{x_\lambda}(\alpha)&=\lambda\, \mathcal Q_{x}(\alpha)+(1-\lambda)r.
	\end{align}
\end{proposition}

\begin{proof}. Denote $f(x,\omega)=\omega' x$, $f(x_\lambda,\omega)=\omega' x_\lambda$. Then
	\begin{align}
		\mathcal F_{x_\lambda}(t)&= \Pr\{f(x_\lambda,\omega) < t\} \\
		  &=\Pr\{ \omega'(1-\lambda)x^0+ \lambda f(x,\omega)< t\} \\
		  &=\Pr\{ (1-\lambda)r+ \lambda f(x,\omega)< t\} \\
		  &=\Pr\{ f(x,\omega)< (t-(1-\lambda)r)/\lambda\} \\
		  &=\mathcal F_{x}\bigl((t-(1-\lambda)r)/\lambda\bigr).
	\end{align}
	Next, by definition, $\mathcal Q_{x_\lambda}(\alpha)$ is the optimal value of the optimization problem
	\begin{align}
		\max\{& t\in \mathbb R\colon \mathcal F_{x_\lambda}(t)\le \alpha\} \\
		&= \max\{t\in \mathbb R\colon \Pr\{f(x_\lambda,\omega)<t \le \alpha\} \\
		&= \max\{t\in \mathbb R\colon \mathcal F_{x}((t-(1-\lambda)r)/\lambda\le \alpha\}.
	\end{align}
	With the variable $\tau= (t-(1-\lambda)r)/\lambda$, the latter problem is equivalent to
	\begin{equation}
		\max\{ \lambda\tau+(1-\lambda)r\colon \mathcal F_{x_\lambda}(\tau)\le \alpha\}.
	\end{equation}
	The optimal value of this problem equals the expression
	\[	\lambda \sup_{\tau\in\mathbb R}\{\tau\colon \mathcal F_x\left(\tau\right)\le\alpha\}+ (1-\lambda)r = \lambda \mathcal Q_{x}(\alpha)+(1-\lambda)r=\mathcal Q_{x_\lambda}(\alpha), \]
	which completes the proof.
%
\end{proof}

The proposition shows that the quantile function of the average portfolio $x_\lambda$ is the similar average of the return quantile functions of the portfolio $x$ and the risk-free portfolio $x^0$.

Thus, for a known risk-free feasible portfolio $x^0= (1,0,\ldots,0)'$ with a fixed return $r>0$, the calculations can be considerably reduced.
Indeed, we need to calculate only one CDF $\mathcal F_{x,m}(\cdot)$ and can re-use it for the CDFs $\mathcal F_{x_\lambda,m}(\cdot)$ for different values of~$\lambda$.
In case of a discrete reference CDF $\mathcal F_\mathit{ref}(t)$ with jumps points
$(\mathcal T_\mathit{ref},{\cal A}_\mathit{ref})=\{(t_1,\alpha_1),\dots,(t_k,\alpha_k),\dots\}$, the projection $p_{H_m}(x)$ can be found in an analytical form as the following proposition states.
\medskip

\begin{proposition}
Assume that $\mathcal Q_\mathit{ref}(1)< r$, i.e., there exists a risk-free internal feasible portfolio $x^0$ with return $r$.
Then the projection $p_{H_m}(x)$ can be stated in the closed form
$p_{H_m}(x)=(1-\lambda_x)x^0+\lambda_x\, x$,
where
\[	\lambda_x=
	\begin{cases}
		1,& \text{if }H_m(x)\le 0,\\
		\min_{\{\alpha\in\mathcal A_\mathit{ref}:\mathcal Q_x<\mathcal Q_\mathit{ref}\}}
	\frac{\mathcal Q_\mathit{ref}(\alpha)-r}{\mathcal Q_{x}(\alpha)-r},&
		\text{if }H_m(x)>0.
	\end{cases}\]
\end{proposition}

\begin{proof}
	For $H_m(x)\le 0$ we have, by definition, $p_{H_m}(x)=x$.
	So we assume that $H_m(x)>0$.
	Consider the portfolios $x_\lambda= (1-\lambda)x^0+\lambda x$, $\lambda\in[0,1]$
	and the function
	\begin{align}
		h_x(\alpha,\lambda)&=\mathcal Q_\mathit{ref}(\alpha)-\mathcal Q_{x_\lambda}(\alpha)\\
		&=\mathcal Q_\mathit{ref}\bigl(\alpha)-r-\lambda(\mathcal Q_{x}(\alpha)-r\bigr).
	\end{align}
	For $\alpha\in\mathcal A_\mathit{ref}$ such that $\mathcal Q_x(\alpha)< \mathcal Q_\mathit{ref}$ it holds that
	$h_x(\alpha,1)>0$ and $h_x(\alpha,0)<0$ and the function $h_x(\alpha,\cdot)$ is linear strictly monotone.
	So the projection corresponds to the minimal $\lambda$ such $h_x(\alpha,\lambda)\ge 0$, that is to $\lambda_x$,
	\[	\lambda_x=\min_{\{\alpha\in\mathcal A_\mathit{ref}\colon \mathcal Q_x<\mathcal Q_\mathit{ref}\}}\frac{\mathcal Q_\mathit{ref}-r}{\mathcal Q_x-r}, \]
	which completes the proof.
\end{proof}

The proposition shows that given the conditions, the feasible set 
$\{x\in X\colon H_m(x)\le 0\}$
has a star shape with respect to the feasible point that represents a risk-free portfolio.

If $\mathcal Q_x(\alpha)$, $\alpha\in{\cal A}_\mathit{ref}$, are continuous functions in~$x$,
then in conditions of the proposition, $\lambda_x$ is continuous, hence the projection mapping
$p_m(x)=\lambda_x x+(1-\lambda_x)x^0$ is also continuous. Then by \cite[Theorem~4.4]{Norkin_2022}
the problems~\eqref{obj2}--\eqref{sdc2} and~\eqref{doubleprojection2} are equivalent.

\section{Numerical Optimization of the Exact Penalty Functions}
\label{sec:Numerical}
In this section, we consider a numerical method for the optimization of generally discontinuous functions $G_m(x)$, $H_m(x)$, $F_m(x)$ in~\eqref{funcG}--\eqref{doubleprojection2}. 
The idea consists in sequential approximations of the original function by smooth (averaged) ones and optimizing the latter by stochastic optimization methods. 
For this we develop stochastic finite-difference estimates of gradients of the smoothed functions.
Although the successive smoothing method has certain global optimization abilities (as discussed in \cite{Norkin_2020}), to strengthen this property we imbed it into some branch and bound scheme as a local optimizer.

\subsection{Averaged Functions}
We limit the consideration to the case of the so-called strongly lower semicontinuous functions.

\begin{definition}[Strongly lower semicontinuous functions, \cite{Ermoliev_Norkin_Wets_1995}]
	A function $F\colon \mathbb R^n\to \mathbb R$ is called 
	\emph{lower semicontinuous (lsc)} at a point $x$, if
	\[	\liminf_{\nu \to \infty} F(x^\nu)\ge F(x)\]
	for all sequences $x^k\to x$.  
	
	A function $F\colon \mathbb R^n\to \mathbb R$ is called \emph{strongly lower semicontinuous} (strongly lsc) at a point $x$, if it is lower semicontinuous at $x$ and there exists a sequence $x^k\to x$ such that it is continuous at $x^k$ (for all $x^k$) and $F(x^k)\to F(x)$.
	A function $F$ is called strongly lower semicontinuous (strongly lower semicontinuous) on $X\subseteq \mathbb R^n$, if this is the case for all $x\in X$.
\end{definition}

The property of strong lower semicontinuity is preserved under continuous transformations.

The averaged functions obtained from the original nonsmooth or discontinuous function by convolution with some kernel have smoother characteristics.
For this reason, they are often used in optimization theory (see \cite{Gupal_Norkin_1977, Ermoliev_Norkin_Wets_1995} and references therein).

\begin{definition}
	The set (family) of bounded and integrable functions
	$\{\psi_\theta \colon \mathbb R^n\to R_+,\ \theta \in \mathbb R_+\}$ satisfying for any $\epsilon>0$ the conditions
	\[	\lim_{\theta\rightarrow 0}\int_{\epsilon\,{\bf B}}\psi_\theta(z)dz=1, \quad
		{\bf B}\coloneqq \{x\in \mathbb R^n\colon \|x\|\le 1\},\]
	is called a family of mollifiers.
	The kernels $\{\psi_\theta\}$ are said to be smooth if the functions $\psi_\theta(\cdot)$ are continuously differentiable.
	
	A function $F\colon \mathbb R^n\to \mathbb R^1$ is called bounded at infinity if there are positive numbers~$C$ and~$r$ such that $|F(x)|\le C$ for all~$x$ with $\|x\|\ge r$.
\end{definition}

Given a locally integrable, bounded at infinity, function $F\colon  \mathbb R^n\to 
\mathbb R^1$ and a family of smoothing kernels $\{\psi_\theta\}$, the associated family of averaged functions $\{F_\theta\colon \theta\in \mathbb R_+\}$ is
\begin{equation}\label{n10}
	F_\theta(x)\coloneqq \int_{\mathbb R^n} F(x+z)\psi_\theta(z)dz
		= \int_{\mathbb R^n} F(z)\psi_\theta(z-x)dz.
\end{equation}

Smoothing kernels can have an unlimited support 
$\operatorname{supp} \psi_\theta= \{x\colon\psi_\theta(x)>0\}$.
To ensure the existence of the integrals~\eqref{n10}, we assume that the function~$F$ is bounded at infinity.
We can always assume this property if we are interested in the behavior of~$F$ within some bounded area.
If $\operatorname{supp} \psi_\theta\to 0$ for $\theta\to 0$, then this assumption is superfluous.

For example, a family of kernels can be as follows.
Let~$\psi$ be some probability density function with bounded support $\operatorname{supp} \psi$, a positive numerical sequence $\{\theta_\nu\colon \nu= 1,2,\dots\}$ tending to~$0$ as $\nu \to \infty$.
Then the smoothing kernels on~$\mathbb R^n$ can be taken as
\[	\psi_{\theta_\nu}(z)\coloneqq \frac1{(\theta_\nu)^n}\psi(z/\theta_\nu).\]

If the function~$F$ is not continuous, then we cannot expect the averaged functions $F_\theta(x)$ to converge to~$F$ uniformly.
But we don't need that.
We need such a convergence of the averaged functions $F_\theta(x)$ to~$F$ that guarantees the convergence of the minima of~$F_\theta(x)$ to the minima of~$F$.
This property is guaranteed by the so-called epi-convergence of functions.

\begin{definition}[Epi-convergence, cf.\ \citet{WetsRockafellar97}]
	A sequence of functions 
	$\{F^\nu\colon \mathbb R^n\to \bar{\mathbb R},\ \nu \in \mathbb N\}$ 
	epi-converges to a function $F\colon \mathbb R^n\to \bar{\mathbb R}$ at a point~$x$, iff
	\begin{enumerate}
		\item $\underset{\nu \to \infty }\liminf F^\nu(x^\nu)\ge F(x)$
		for all $x^\nu\to x$;
		\item $\underset{\nu \to \infty}\liminf F^\nu(x^\nu)=F(x)$
			for some sequence $x^\nu\to x$. The sequence $\{F^\nu\}$ epi-converges to $F$, if this is the case at every point $x\in \mathbb R^n$.
	\end{enumerate}
\end{definition}

\begin{theorem}	[Epi-convergence of avraged functions, cf.\ \cite{Ermoliev_Norkin_Wets_1995}] \label{averconvergence}
	For a strongly semicontinuous locally integrable function
	$F\colon\mathbb R^n\to \mathbb R^1$, any associated sequence of averaged functions 
	$\{F^\nu=F_{\theta_\nu}\colon \theta_\nu \in \mathbb R_+\}$ epi-converges to~$F$ as 
	$\theta_\nu\downarrow 0$.
\end{theorem} 

Note that in the optimization problem without constraints the theorem states that $\lim_\nu\left( \inf_xF^\nu \right)=\inf_xF$.

To optimize discontinuous functions, we approximate them with averaged functions.
The convolution of a discontinuous function with the corresponding kernel (probability density) improves analytical properties of the resulting function, but increases the computational complexity of the problem, since it transforms the deterministic function into an expectation function, which is a multidimensional integral.
Therefore, such an approximation makes sense only in combination with the corresponding stochastic optimization methods.
First, we consider conditions of continuity and continuous differentiability of the averaged functions.

We can also consider smoothed functions obtained by employing differentiable kernel with unbounded support. For example, let the kernel be the Gaussian probability density, i.e.,
\[	\psi(y)= (2\pi)^{-n/2}\,e^{-\|y\|^2/2}.\]
Let us consider the family
\[	F_\theta(x)=\int_{\mathbb R^n}F(x+\theta y)\,\psi\left(y\right)dy 
=\frac{1}{\theta^n}\int_{\mathbb R^n}F(z)\,\psi\left(\frac{z-x}\theta \right)dz,\quad \theta >0,\]
of averaged functions.
Suppose that~$F$ is globally bounded (one may even assume that $|F(x)|\le \gamma_1+\gamma_2\|x\|^{\gamma_3}$ with some non-negative constants $\gamma_1$, $\gamma_2$ and $\gamma_3$).
Then for the strongly lsc function~$F$, the average functions $F_\theta$ epi-converge to~$F$ as $\theta\downarrow 0$ and each function $F_\theta$ is analytical with gradient
\begin{eqnarray}
	\nabla F_\theta(x)&=&
	\frac{1}{\theta^{n+2}}\int_{\mathbb R^n}F(z)\psi \left(\frac{z-x}\theta \right)(z-x)\,dy 
	=\frac{1}\theta\int_{\mathbb R^n}F(x+\theta y)\,\psi(y)\,y\,dy\\
	&=&- \frac{1}\theta\int_{\mathbb R^n}F(x-\theta y)\,\psi(y)\,y\,dy
	=\frac{1}\theta\int_{\mathbb R^n}[F(x+\theta y)-F(x)]\,\psi(y)\,y\,dy=\\
	&=& \frac{1}{2\theta}\int_{\mathbb R^n}[F(x+\theta y)-F(x-\theta y)]\,\psi(z)\,z\,dz
\end{eqnarray}
or
\begin{equation} \label{n13}
	\nabla F_\theta(x)=\E_\eta\frac{1}{\theta}[F(x+\theta \eta)-F(x)]\eta
	= \E_\eta\frac{1}{2\theta}[F(x+\theta \eta)-F(x-\theta\eta)]\eta,
\end{equation}
where the random vector $\eta$ has the standard normal distribution and $\E_\eta$ denotes the mathematical expectation over $\eta$.
Thus, the random vector
\begin{equation}\label{n14}
	\xi_\theta(x,\eta) =\frac{\eta}{2\theta}[F(x+\theta\eta) - F(x-\theta\eta)]
\end{equation}
with the (Gaussian) random variable~$\eta$ is an unbiased statistical estimate of the gradient $\nabla F_\theta(x)$.

\subsection{Stochastic Methods for Minimization of Discontinuous Penalty Functions} \label{sssm}
Consider a problem of constrained minimization of a generally discontinuous function 
subject to a box or other convex constraints.
The target problems are~\eqref{funcG}--\eqref{doubleprojection2}.
 
Such problems can be solved, e.g., by collective random search algorithms. 
In this section we develop stochastic quasi-gradient algorithms to solve these problems. 

A problem of constrained optimization can be reduced to the problem of unconstrained optimization of a coercive function~$F(x)$ by using nonsmooth or discontinuous penalty functions as described in \cite{Galvan_et_al_2021, Norkin_2020, Norkin_2022} (for the case of the present paper see Section~\ref{sec:exactPF}). 

Suppose the function~$F(x)$ is strongly lower semicontinuous.
In view of Theorem~\ref{averconvergence}
it is always possible to construct a sequence of smoothed averaged functions $F_{\theta_\nu}$ that epi-converges to~$F$.
Due to this property, 
global minima of $F_{\theta_\nu}$ converge to the global minima of~$F$ as $\theta_\nu\to 0$.
Convergence of local minima was studied in~\cite{Ermoliev_Norkin_Wets_1995}.

Let us consider some procedures for optimizing function $F$  using approximating averaged functions $F_{\theta_\nu}$.

Suppose one can find the global minima $\{x^\nu\}$ of functions $F_{\theta_\nu}$, $\nu= 0,1,\ldots$.
Then any limit point of the sequence $\{x^\nu\}$ is a global minimum of the function~$F$.
However, finding global minima of $F_{\theta_\nu}$ can be a quite difficult task, so consider the following method.

\paragraph{{\bf The Successive Stochastic Smoothing Method}, cf.\ \cite{Norkin_2020}}
The method sequentially minimizes a sequence of smoothed functions $F_{\theta_\nu}$ with decreasing smoothing parameter $\theta_\nu\downarrow 0$.
Here, the sequence of approximations $x^\nu$ is constructed by implementing the following steps (cf.\ \cite{Norkin_2020}).
\smallskip

\emph{The successive smoothing method as a local optimizer.}
\begin{enumerate}
	\item Fix a decreasing sequence of smoothing parameters $\{\theta_\nu\}$ with sufficiently large initial value $\theta_1$.
	Select a starting point $x^0$.
	\item \label{enu:ii} For a fixed smoothing parameter $\theta_\nu$, $\nu\ge 1,$ minimize the smoothed function $F_{{\theta }_\nu}$ by some stochastic optimization method with the use of the initial point $x^{\nu-1}$ and finite-difference stochastic gradients~\eqref{n14} to find the next approximation $x^\nu$. 
	\item Set $\nu\coloneqq\nu+1$ and return to step~\ref{enu:ii} unless a stopping criterion is fulfilled.
\end{enumerate}
\medskip

For the minimization of the smoothed function $F_{\theta_\nu}$ under fixed $\theta_\nu$, one can apply any stochastic finite-difference optimization method based on the finite difference representations~\eqref{n13}, \eqref{n14} of the gradients of the smoothed functions. 
Some stopping rules for stochastic gradient methods are discussed in \cite{Pflug_1988}.
Asymptotic convergence of such methods to critical points of $F_{\theta_\nu}(x)$ was studied in \cite{Gupal, Mikhalevich_Gupal_Norkin_1987, Polyak_1987}.
If $\nabla F_{\theta_\nu}(x^\nu)\rightarrow 0$, then, by results of \cite{Ermoliev_Norkin_Wets_1995}, the constructed sequence $x^\nu$ asymptotically converges to the set, which satisfies necessary optimality conditions for~$F$.
If this method is applied sequentially to a sequence of smoothed functions $F_{\theta _\nu}$ with $\theta_\nu\downarrow 0$  it can approach to the global minima of~$F$ (cf.\ \cite{Norkin_2020}).

This method requires estimating gradients $\nabla F_{\theta_\nu}(x)$ during the iterative optimization process for a smooth averaged function $F_{\theta_\nu}$ to set up stopping rules.
In general, this is a rather complicated and time-consuming procedure that requires  calculation of multidimensional integrals.
However, such asymptotically consistent estimates can be constructed in parallel with the construction of the main minimization sequence by using the following so-called averaging procedure \cite{Ermoliev_1976, Gupal, Mikhalevich_Gupal_Norkin_1987}.

Consider the stochastic optimization procedure
\begin{align}
	x^{k+1}&= x^k- \rho_k\, z^k,\quad z^0= \xi_0(x^0),\quad x^0\in \mathbb R^n,\\
	z^{k+1}&= z^k- \lambda_k\big(z^k- \xi_k(x^k, \eta^k)\big),\quad k=0,1,\ldots
\end{align}
for iterative optimization of a function $F_\theta(x)$ and parallel evaluation of its gradients $\nabla F_\theta(x)$,
where the vectors $\xi_k(x^k,\eta)$ are given by~\eqref{n14}. For the conditional expectations it holds that
$\E \bigl(\xi_k(x^k,\eta^k)\mid x^k\bigr)= \nabla F_\theta(x^k)$.
Let numbers $\rho_k$, $\lambda_k$ satisfy conditions
\[	0\le \lambda_k\leq 1,\ \ \lim_k\lambda_k= 0,\ \
\sum_{k=0}^\infty\lambda_k= +\infty,      
\ \ \sum\limits_{k=0}^\infty\lambda_k^2< +\infty,\ \
	\underset{k}{\lim}\,\frac{\rho_k}{\lambda_k} =0.\]
Then, with probability one, it holds (cf.\ \citet[Theorem~V.8]{Ermoliev_1976}) that
\[	(z^k- \nabla F_\theta(x^k)\to 0 \text{ as }k\to \infty.\]

In practical calculations we use a certain version of the successive smoothing method (cf.\ \cite{Norkin_2020}) with the finite-difference stochastic gradients~\eqref{n13} and~\eqref{n14}. 
In the method, we do not estimate the gradients $\nabla F_{\theta_\nu}(x^\nu)$ but just fix a number of smoothing steps~$N$ and take diminishing smoothing parameters $\theta_\nu= \theta_1(1-(\nu-1)/N)$, $\nu=1,\ldots,N$.
For each fixed parameter $\theta_\nu$ we allocate some fixed number (a portion of~$N$, e.g., $\approx\sqrt N$) of steps of a chosen stochastic gradient method, e.g., Nemirovski-Yudin method with the trajectory averaging \cite{Nemirovsky_Yudin_1983, Nemirovski2009}.
To mitigate the influence of discontinuities, we normalize decent directions of the stochastic gradient method.
The transition between smoothing stages is done by a Gelfand-Tsetlin-Nesterov step \cite{Gelfand_Tsetlin_1962, Nesterov_1983}.
Figure~\ref{fig1} gives a graphical illustration of the performance of the successive smoothing method on problem~\eqref{funcF2}.

\begin{figure*}[ht]
	\includegraphics[width=0.49\textwidth]{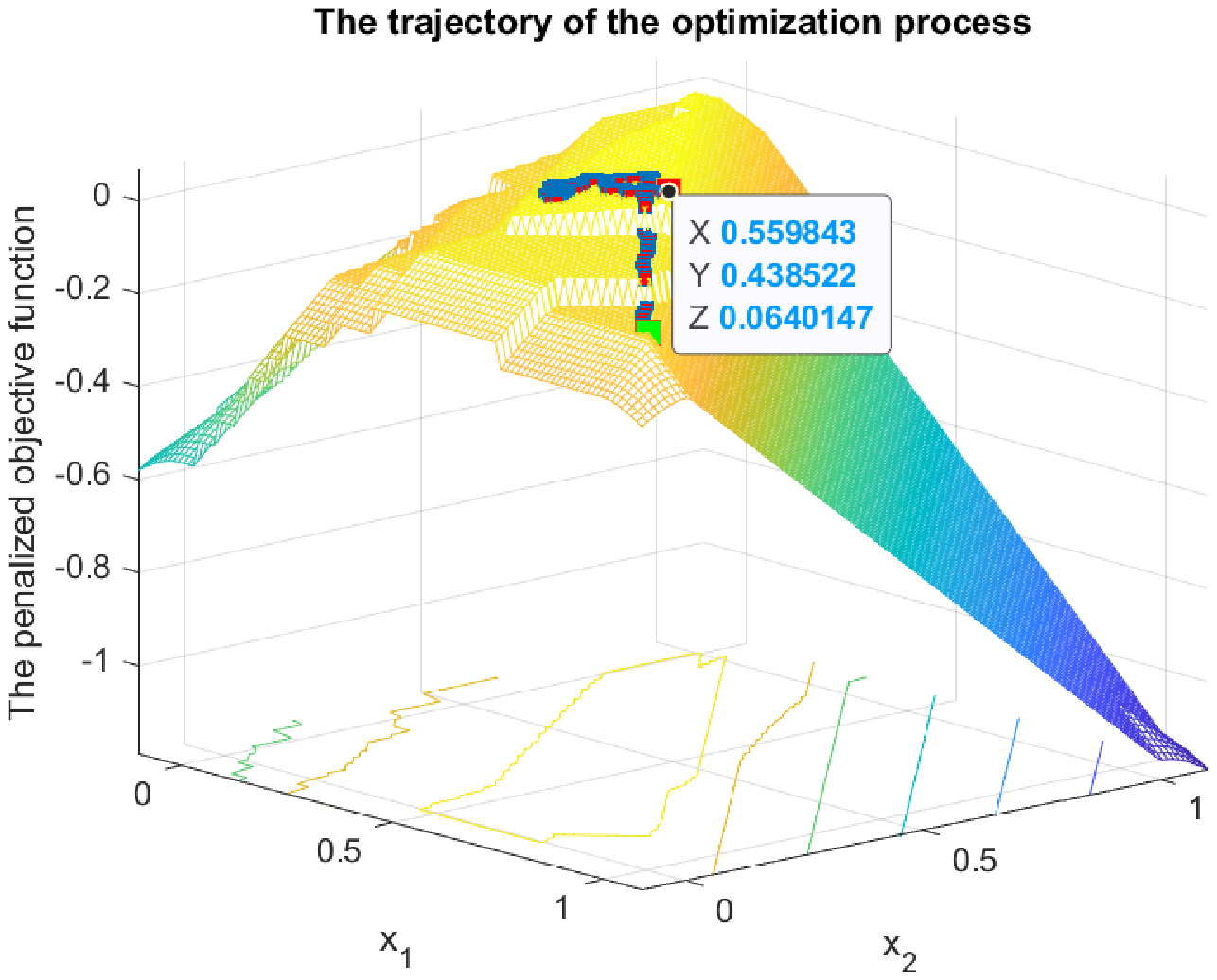}
	\includegraphics[width=0.49\textwidth]{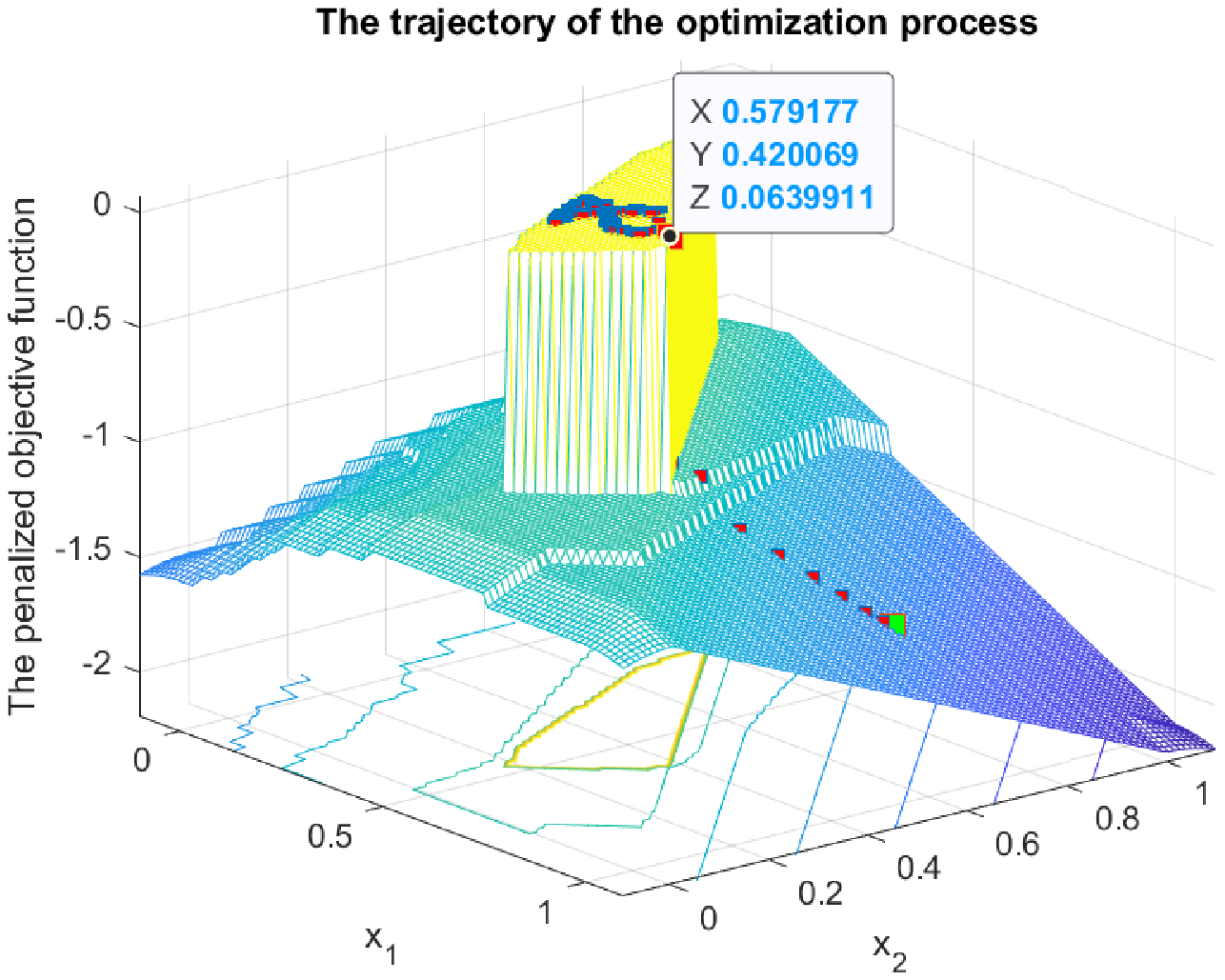}
	\caption{\label{fig1}Illustration of the smoothong method performance on two asset (1,2) portfolio selection. Examples of the trajectories of the method for different discontinuous penalties. 
$x_\mathit{ref}=(x_1=0.3,x_2=0.7)$, $\delta=0.05$.}
\end{figure*}

\paragraph{\bf Global optimization issues} The optimization problems under consideration~\eqref{obj1}\nobreakdash--\eqref{sdc1} and~\eqref{obj2}--\eqref{sdc2} are challenging, nonconvex and multi-extremal, discontinuous and constrained.
So we apply different techniques to solve them.

To remove discontinuous constraints, we use exact discontinuous penalty functions, and to remove structural portfolio constraints, we apply exact projective penalty functions, cf.\ \cite{Galvan_et_al_2021, Norkin_2020, Norkin_2022}.

In the case of multi-extremal problems (as in Figure~\ref{fig2}) we employ a version of the branch and bound method \cite{Norkin_2022} with the successive smoothing method \cite{Norkin_2020} as a local minimizer (in \cite{Norkin_2022} the successive quadratic approximation method was used as a local optimizer).

To solve problem~\eqref{funcG}--\eqref{doubleprojection2}, we apply the following branch and bound (cut) algorithm.
\medskip

\paragraph{\bf The Branch \& Bound algorithm}
\begin{description}
	\item[Initialization.]
	Set the initial partition $\mathcal P_0=\{X\}$, select a random starting point $\tilde x^0\in X$ and apply some \emph{local optimization algorithm} $\mathcal A$ to the problem under consideration.
	As result we find a better point $\bar x^0\in X$ such that
	$F(\bar x^0)<F(\tilde x^0)$.
	Set the B\&B iteration count $k=0$. Set tolerances $\epsilon>0$ and $\delta>0$.

	\item[B\&B iteration.]Suppose at iteration~$k$ we have partition $\mathcal P_k= \{X_i\colon i=1,\dots,N_k\}$ of the set  $X=\cup_{i=1}^{N_k}X_i$ consisting of smaller boxes $X_i$.
	For each $X_i$, there is a known feasible point $\bar x^i\in X_i$ and the corresponding value $F(\bar x^i)$, $V_k=\min_{1\le i\le N_k}F(\bar x^i)$.
	Set $\mathcal P_{k+1}= \emptyset$.\\[0.4em]
	For each such set $X_i\in \mathcal P_k$
	choose  a random starting point $\tilde x^i$ and apply some \emph{local optimization algorithm} ${\cal A}$ to the problem $\min_{x\in X_i}F(x)$ to find a better point $\bar{\bar x}^i\in X_i$,
	$F(\bar\bar x^i)< F(\tilde x^i)$.\\[0.4em]
	If the values $F(\bar x^i)$ and $F(\bar{\bar x}^i)$ are sufficiently different, say $\|F(\bar{x}^i)-F(\bar{\bar{x}}^i)\|\ge\epsilon$, or
	the points $\bar{\bar{x}}^i$ and $\bar{x}^i$ are sufficiently distinct, 
	$\|\bar{\bar{x}}^i-\bar{x}^i\|\ge\delta$, we subdivide the box 
	$X_i= X^\prime_i\cup X^{\prime\prime}_i$ into two subboxes
	$X^\prime_i$ and $X^{\prime\prime}_i$ so that
	$\bar x^i\in X^\prime_i$ and $\bar{\bar x}^i\in X^{\prime\prime}_i$. 
	In this case, the partition $\mathcal P_{k+1}$ is updated by adding the successors 
	$X^\prime_i$ and $X^{\prime\prime}_i$, i.e.,
	Otherwise, if the values $F(\bar x^i)$ and $F(\bar{\bar x}^i)$ and points
	$\bar{\bar x}^i$ and $\bar x^i$ are close, the set $X_i\ni \bar x^i$ goes unchanged to the updated partition $\mathcal P_{k+1}\coloneqq\mathcal P_{k+1}\cup X_i$.\\[0.4em]
	When all elements $X_i\in \mathcal P_k$ are checked, i.e., the new partition 
	$\mathcal P_{k+1}$ with elements $X_i, i=1,\dots,N_{k+1}$, and points $\bar x_i\in X_i$
	has been constructed, we modify the value achieved,
	$V_{k+1}=\min_{1\le i\le N_{k+1}} F(\bar x^i)$.
	\item[Check for stop.] If the progress of the B\&B method becomes small, e.g., $V_k-V_{k+1}$ 
		(or $V_{k-1}-V_{k+1}$, etc.) is sufficiently small, otherwise, repeat the B\&B iteration.
\end{description}
\medskip

\begin{remark}
	The function values $F(\bar x^i)$ of the objective provide upper bounds for the optimal values
	$F^*_i=\min_{x\in X_i}F(x)$.
	If there are known lower bounds $L_i\le F^*_i$, then the subsets $X_i\in\mathcal P_k$ such that $L_i\ge V_k$ can be safely ignored, i.e., excluded from the current partition $\mathcal P_k$.
	Heuristically, if some set $X_i$ remains unchanged during several {\it B\&B} iterations, it can be ignored in the future iterations.
	Further results of the B\&B algorithm described above are available in \cite{Norkin_2022}.
\end{remark}
 
\section{Numerical illustration}
For the numerical illustration of the algorithm proposed we return to portfolio optimization under 1\textsuperscript{st} order stochastic dominance constraints. We use a small data set of annual returns of nine US companies from \citet[Table~1, page~13]{Markowitz_1959} (see the appendix).

\begin{figure*}[ht]
	\includegraphics[width=0.49\textwidth]{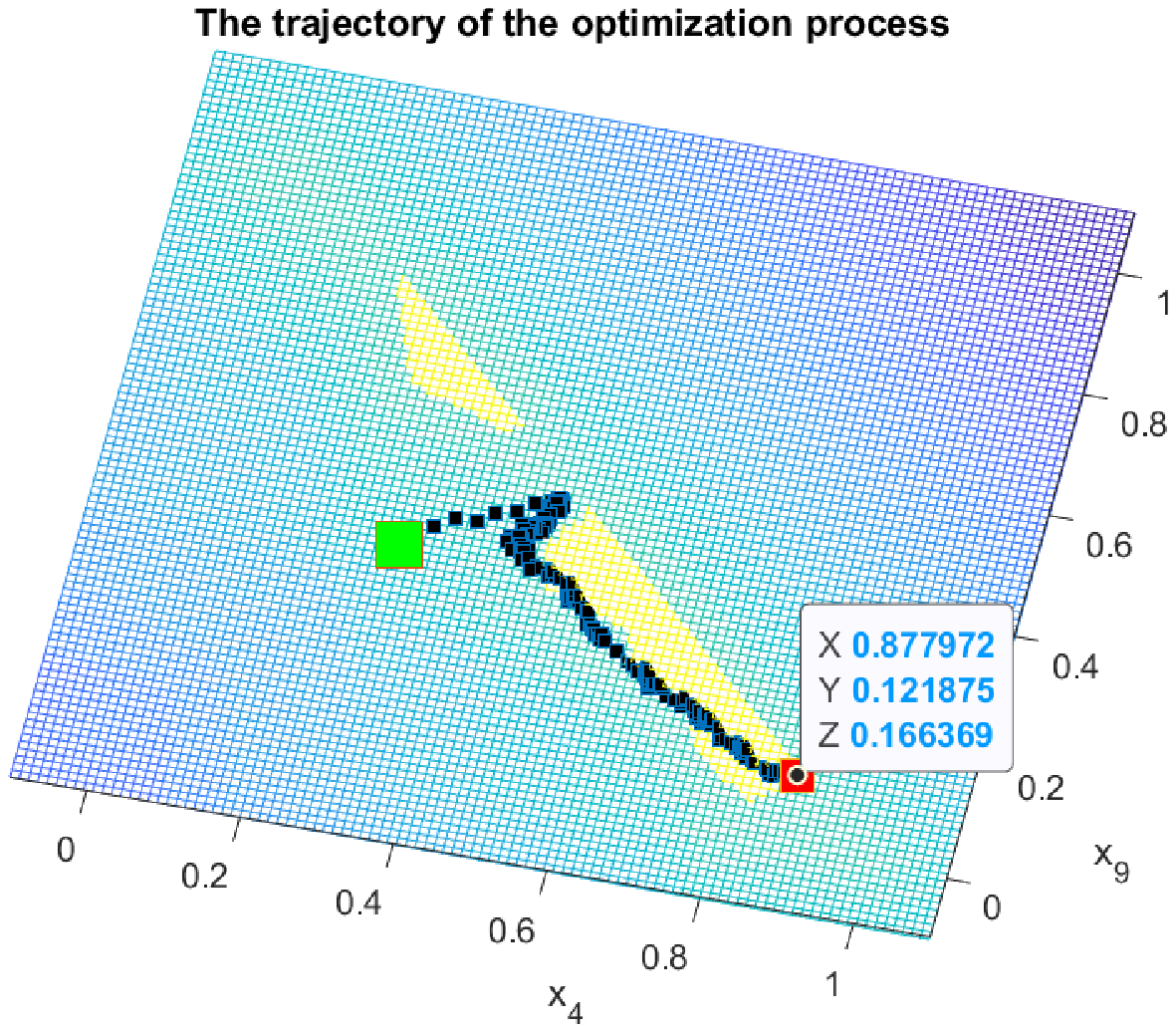}
	\includegraphics[width=0.49\textwidth]{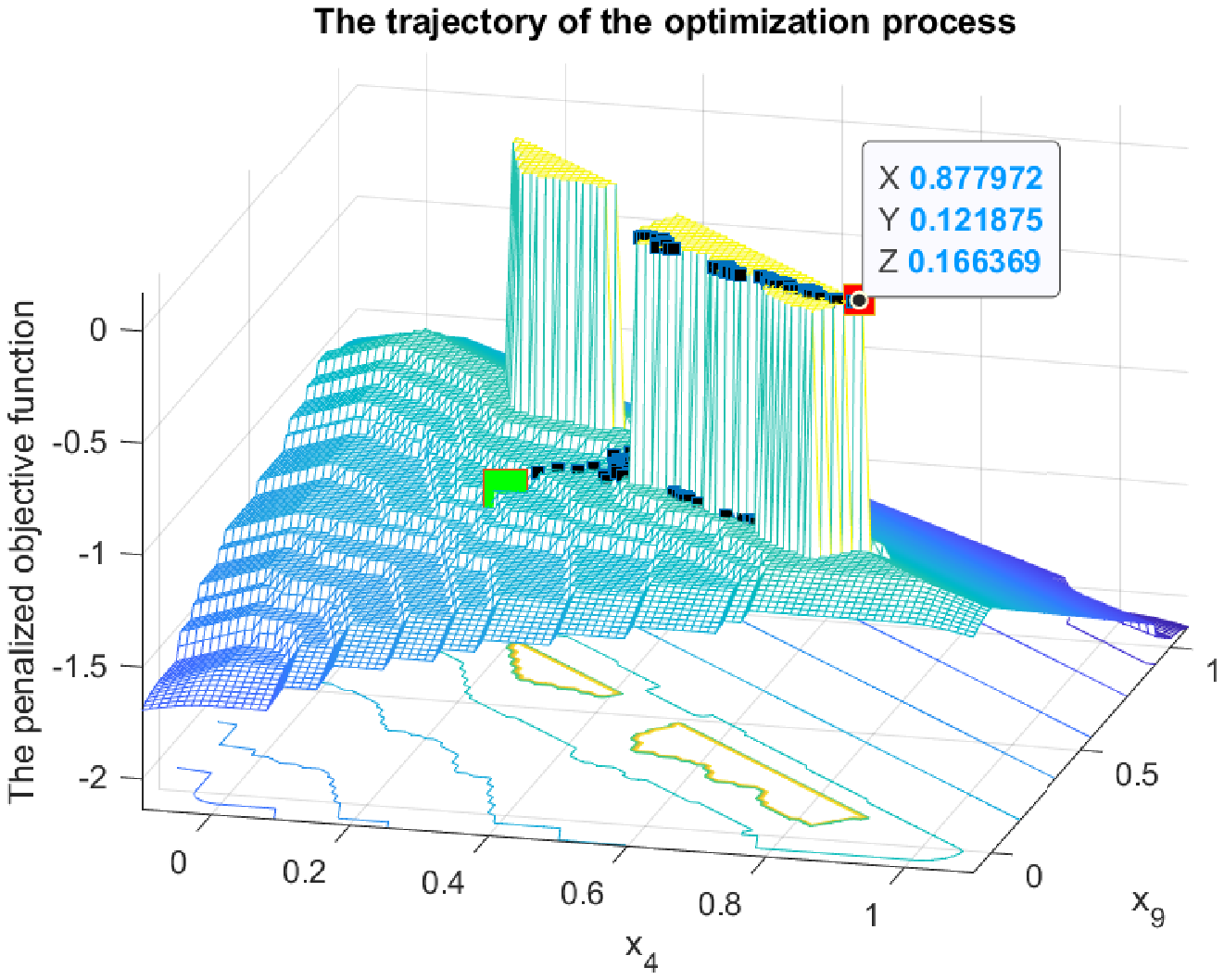}
	\caption{Illustration of the smoothing method performance on two asset (4,9) portfolio selection. Example of the trajectory of the method for non-connected feasible set.
	$x_\mathit{ref}=(x_4=0.3,x_9=0.7)$, $\delta=0.05$. The green is the starting point, the red is the final one.}
	\label{fig2}       
\end{figure*}

\subsection{Testing the successive smoothing method on the discontinuous portfolio optimization problems}
First, let us illustrate the proposed approach on a two-dimensional portfolio (the first two columns of the table from the appendix) by solving the problem~\eqref{obj1}--\eqref{sdc1}.
For this we first fix some initial reference portfolio, $x_\mathit{ref}'=(0.7, 0.3)$ with average return $\mu_\mathit{ref}=0.0629$ and corresponding CDF $\mathcal F_\mathit{ref}(t)$.
Then we relax the constraint~\eqref{sdc1} by replacing $\mathcal F_\mathit{ref}(t)$ with shifted function $\mathcal F_\mathit{ref}(t-\delta)$, $\delta=0.05$.
Note, that we can take $\delta$ dependent on $t$.
With this new reference function we solve the problem~\eqref{funcF} with different values of the parameter $c\in \{0.0659,1.0659\}$ by the stochastic smoothing method from Subsection~\ref{sssm}. 

Figure~\ref{fig1} presents the results.
The pictures illustrate how the method climbs up to the global maximum of the discontinuous objective function.
In the two presented examples the method finds better portfolios with average return~$\mu\approx0.0640$.
The right picture also highlights the set of feasible portfolios because by setting a larger~$c$, we let down the values of the penalized function at the infeasible points.
It can be seen that the proposed version of the smoothing method is not very sensible to
discontinuities of the minimized function.

One more example is presented on Figure~\ref{fig2}.
The reference 2-security portfolio is $x_\mathit{ref}'=(x_4=0.3, x_9=0.7)$, $\delta=0.05$.
In this example, the feasible set is not connected.
The optimal portfolio consists of the two assets $x_{4,9}^*=(x_4=0.8779, x_9=0.1219)'$ with the expected return $\mu(x_{4,9}^*)=0.1664$.
If we extend the portfolio to nine assets, we can obtain by the proposed method the better return $\mu_{1:9}^*=0.1724$ with the optimal portfolio 
\[	x^*=( 0.0117, 0.0131, 0.0730, 0.2936, 0.4619, 0.0123, 0.0080, 0.0324, 0.0880)',\]
$\sum_i x^*_i= 0.9941$, within the same bounds on risk, $G(x^*)=0$.

\subsection{Lower bounding the risk-return profile and maximizing the tail return}
In this subsection we present results of optimization of 3- and 10-component portfolios
under 1\textsuperscript{st} order stochastic dominance constraints. 
The constraint is given by its (reference) CDF, and as objective functions we use the average value (AV) of the portfolio return, the Value-at-Risk (quantile, $\VaR_\alpha$) and Average Value-at-Risk indicators ($\AVaR_\alpha$) with levels $\alpha= 40\,\%$ and $\alpha= 70\,\%$.
The results are given in table and in graphical forms, each set of experiments is specified by the number of portfolio components (3 or 10), value of parameter $\alpha$ ($40\,\%$ or $70\,\%$) and the exact penalty method applied (discontinuous or projective from Subsections~\ref{edpf} and~\ref{eppf}).
So each table contains three numerical rows corresponding to three kinds of the objective functions: (1) the mean, (2) the $\VaR$, and (3) the $\AVaR$.
The columns ‘Mean’, ‘$\VaR$’ and ‘$\AVaR$’ show the objective values of the corresponding indicators for the three optimal portfolios.
The other columns show the structure of the obtained portfolios.
Each table is supplemented by a reference figure containing three graphs displaying the optimal risk profile.
The blue broken line in a graph depicts the reference CDF, a (left) bound on the portfolio return, CDF([0.05; 0.05; 0.1; 0.11; 0.125])=[0; 0.2; 0.4; 0.6; 1]. 
The red (right) broken line shows the CDF of the actual optimal portfolio return.
So the lines display the reference and the actual risk profiles of the optimal portfolios. 

The Tables~\ref{table_2022_12_03_ind_3comp_disc} and~\ref{table_2022_12_03_ind_3comp_anal_lev0.4}
(and the corresponding Figures~\ref{fig.2022.12.03_3comp_disc} and~\ref{fig.2022.12.03_3comp_anal_lev0.4}) compare results of solving portfolio optimization problems by means of discontinuous and projective penalty methods, respectively, $\alpha= 40\,\%$.
As can be seen, both figures, Figure~\ref{fig.2022.12.03_3comp_disc} and~\ref{fig.2022.12.03_3comp_anal_lev0.4}, are very similar, that can be a proof that both penalty methods are applicable and give close results.

\begin{table*}[ht]
	\caption{Optimal 3 component portfolios, discontinuous penalties: 1)~Max mean. 2)~Max $\VaR$. 3)~Max $\AVaR$. See Fig.~\ref{fig.2022.12.03_3comp_disc}.}
\label{table_2022_12_03_ind_3comp_disc}
\begin{tabular}{rrrrrrr}
\hline
Portf. No	&Mean	&$\VaR_{0.4}$	&$\AVaR_{0.4}$	&\;$x_4$	&$x_9$	&$x_{10}$	\\
1	&0.1351	&0.1320 &0.1552 &0.1229	&0.0085	&0.8675	\\
3	&0.1355 &0.1344	&0.1560	&0.1313	&0.0001	&0.8670 \\
2	&0.1351	&0.1323	&0.1554	&0.1219	&0.0105	&0.8675 \\
\hline
\end{tabular}
\end{table*}

\begin{figure*}[ht]
  \includegraphics[width=0.325\textwidth]{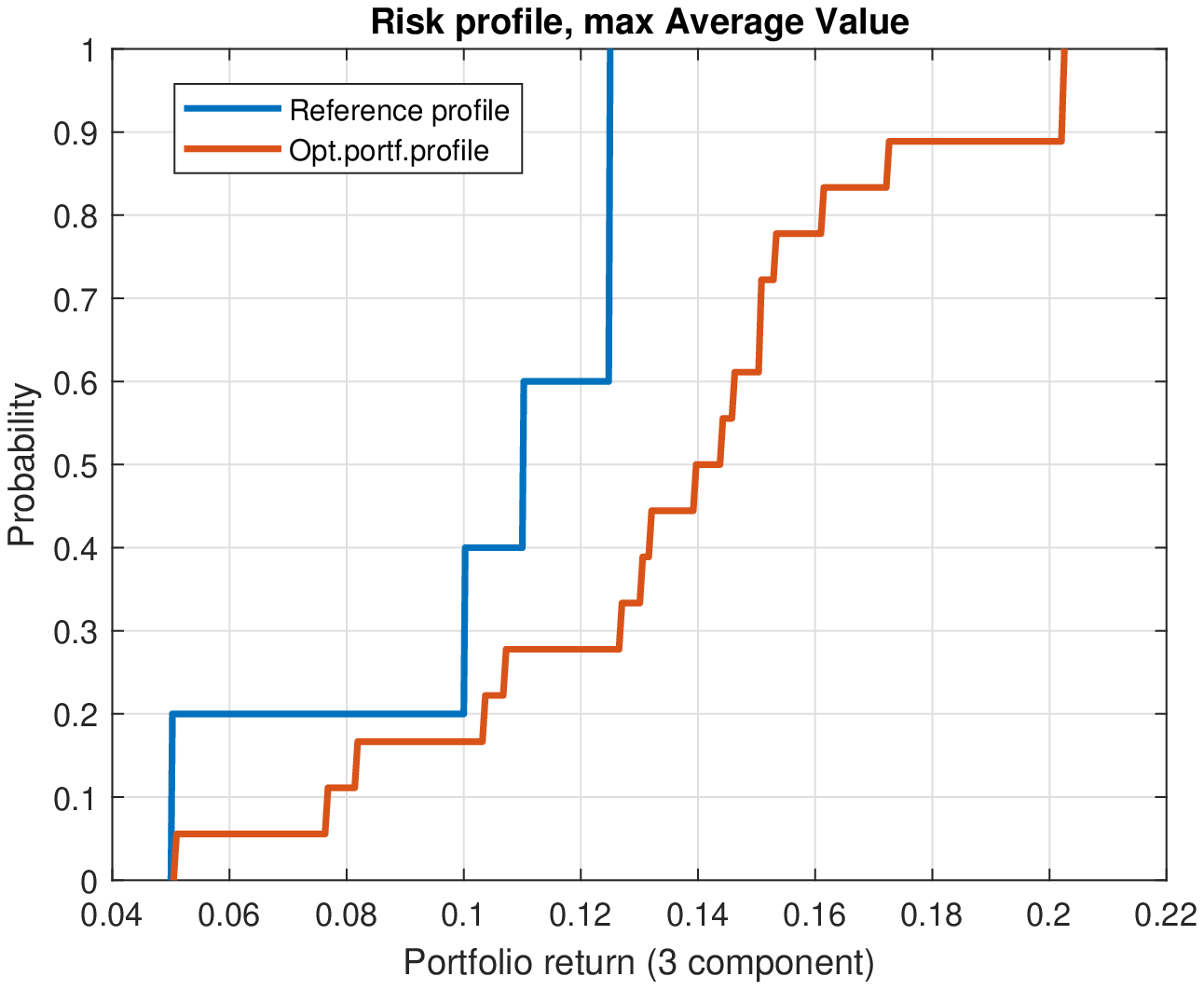}
	\includegraphics[width=0.335\textwidth]{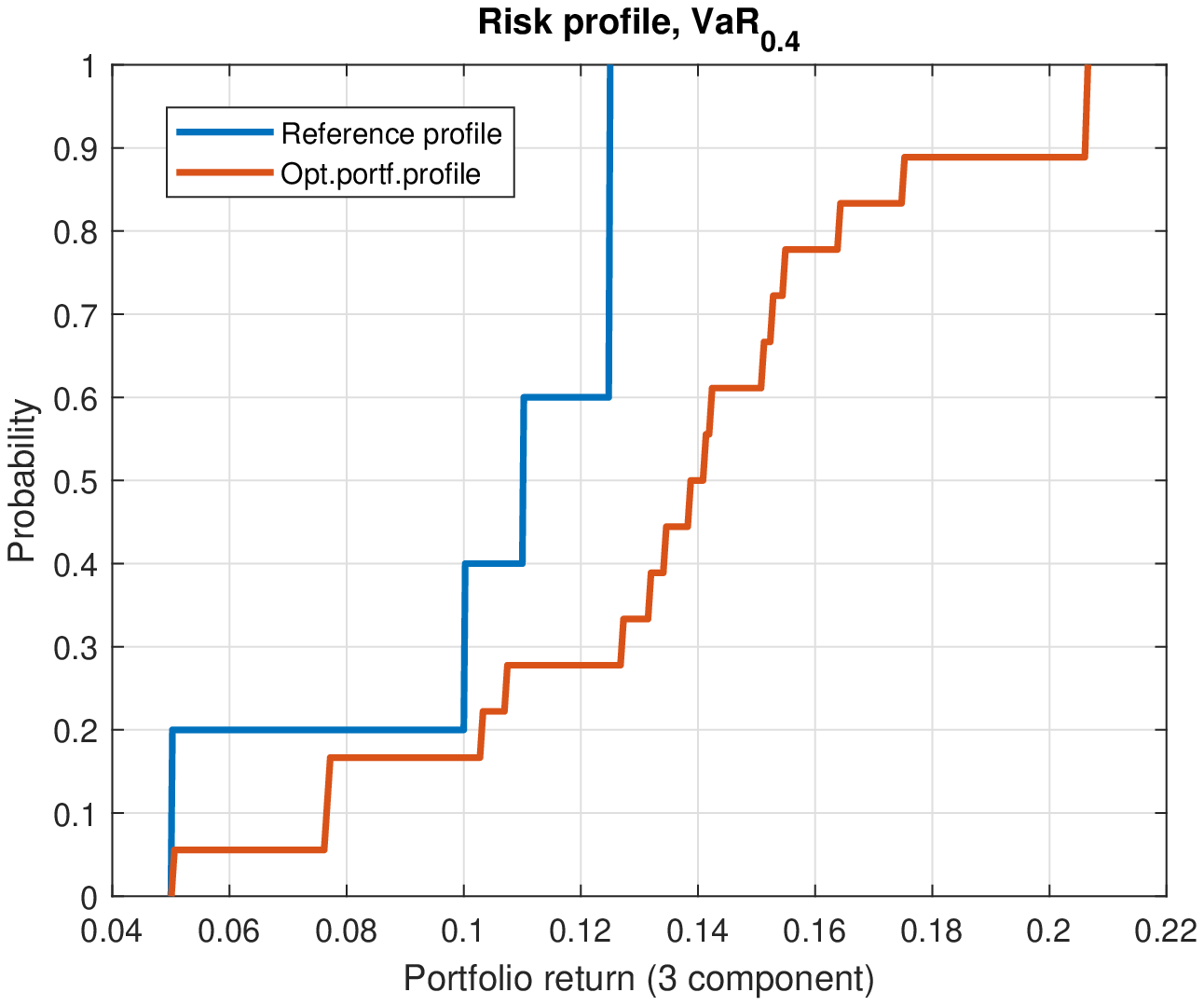}
	\includegraphics[width=0.325\textwidth]{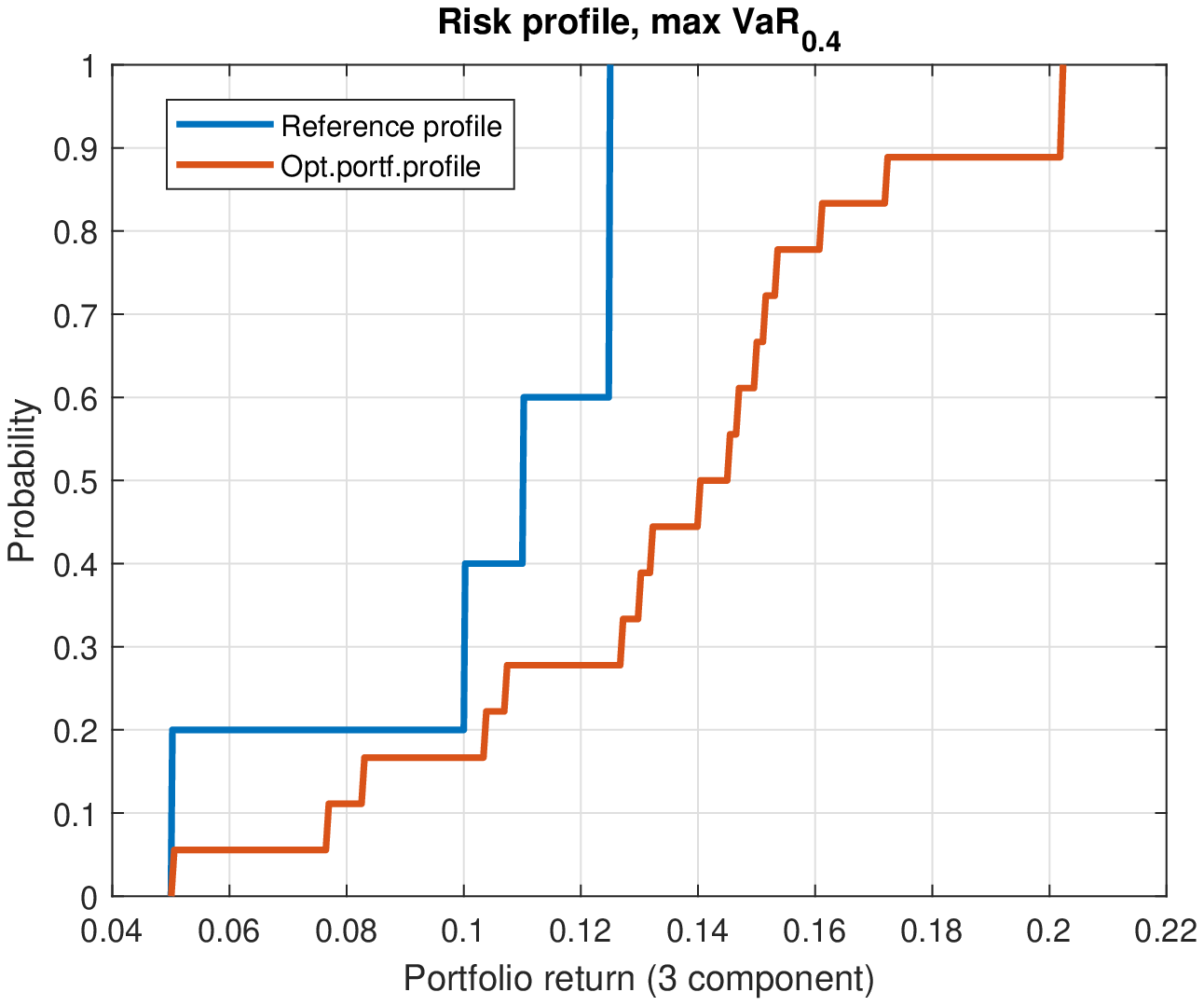}
	
\caption{Profiles of optimal 3 component portfolios: maximizing the 
tail returns under the risk-return lower bound. Discontinuous penalties.
1) Optimal average return. 
2)  Optimal $\VaR_{40\%}$.  
3)  Optimal $\AVaR_{40\%}$; cf.\ Table~\ref{table_2022_12_03_ind_3comp_disc}.
}
\label{fig.2022.12.03_3comp_disc}       
\end{figure*}

\begin{table*}[ht]
\caption{Optimal 3 component portfolios, analytical projection: 
1)~Max mean. 2)~Max $\VaR_{40\%}$. 3)~Max $\AVaR_{40\%}$. See Fig.~\ref{fig.2022.12.03_3comp_anal_lev0.4}.}
\label{table_2022_12_03_ind_3comp_anal_lev0.4}
\begin{tabular}{rrrrrrr}
\hline
Portf. No	&Mean	&$\VaR_{0.4}$	&$\AVaR_{0.4}$	&\;$x_4$	&$x_9$	&$x_{10}$	\\
1	&0.1357	&0.1531 &0.1707 &0.1308	&0.0009	&0.8683	\\
2	&0.1356	&0.1527	&0.1704	&0.1298	&0.0000	&0.8702 \\
3	&0.1357 &0.1531	&0.1709	&0.1313	&0.0004	&0.8684 \\
\hline
\end{tabular}
\end{table*}

\begin{figure*}[ht]
  \includegraphics[width=0.325\textwidth]{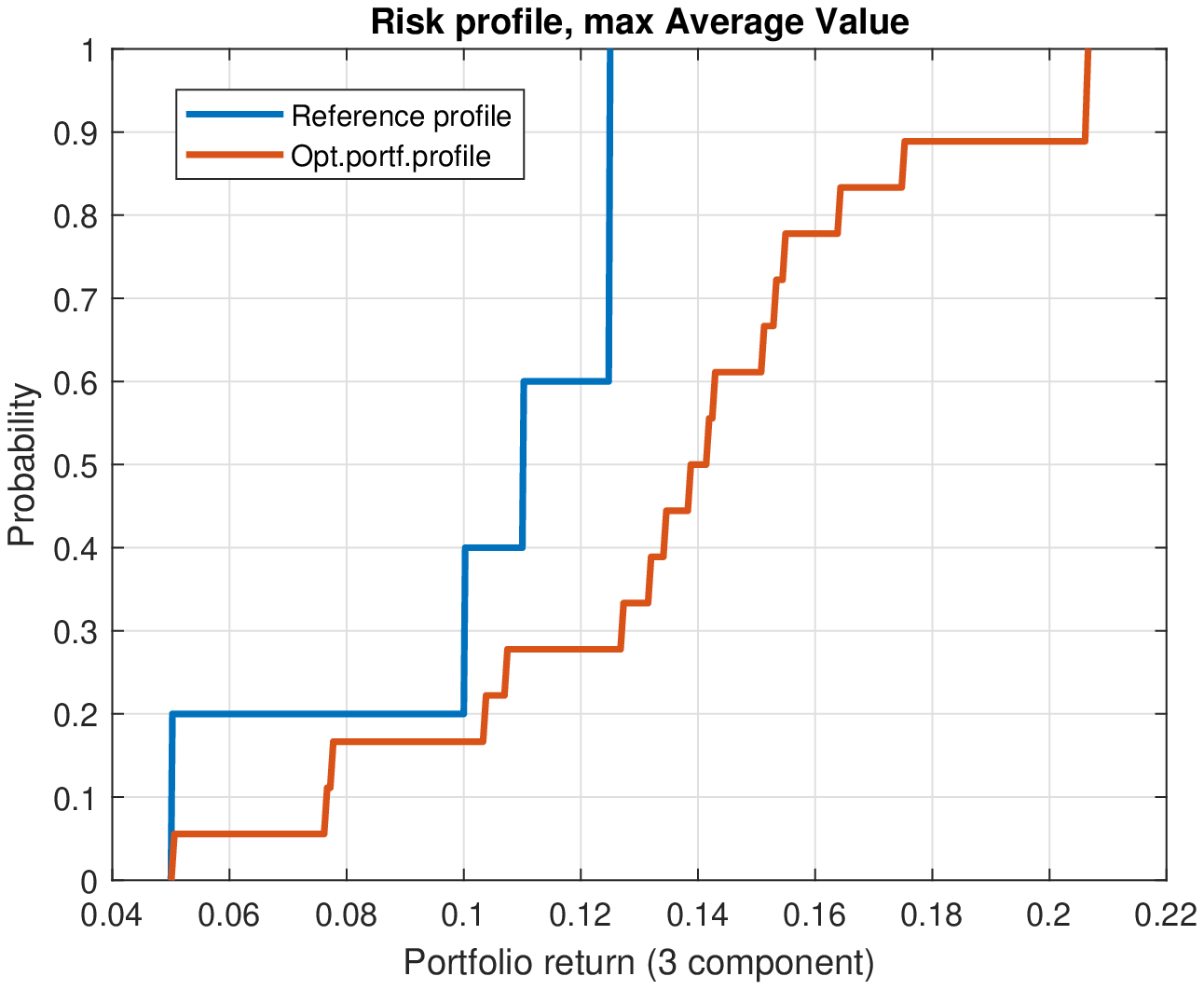}
	\includegraphics[width=0.325\textwidth]{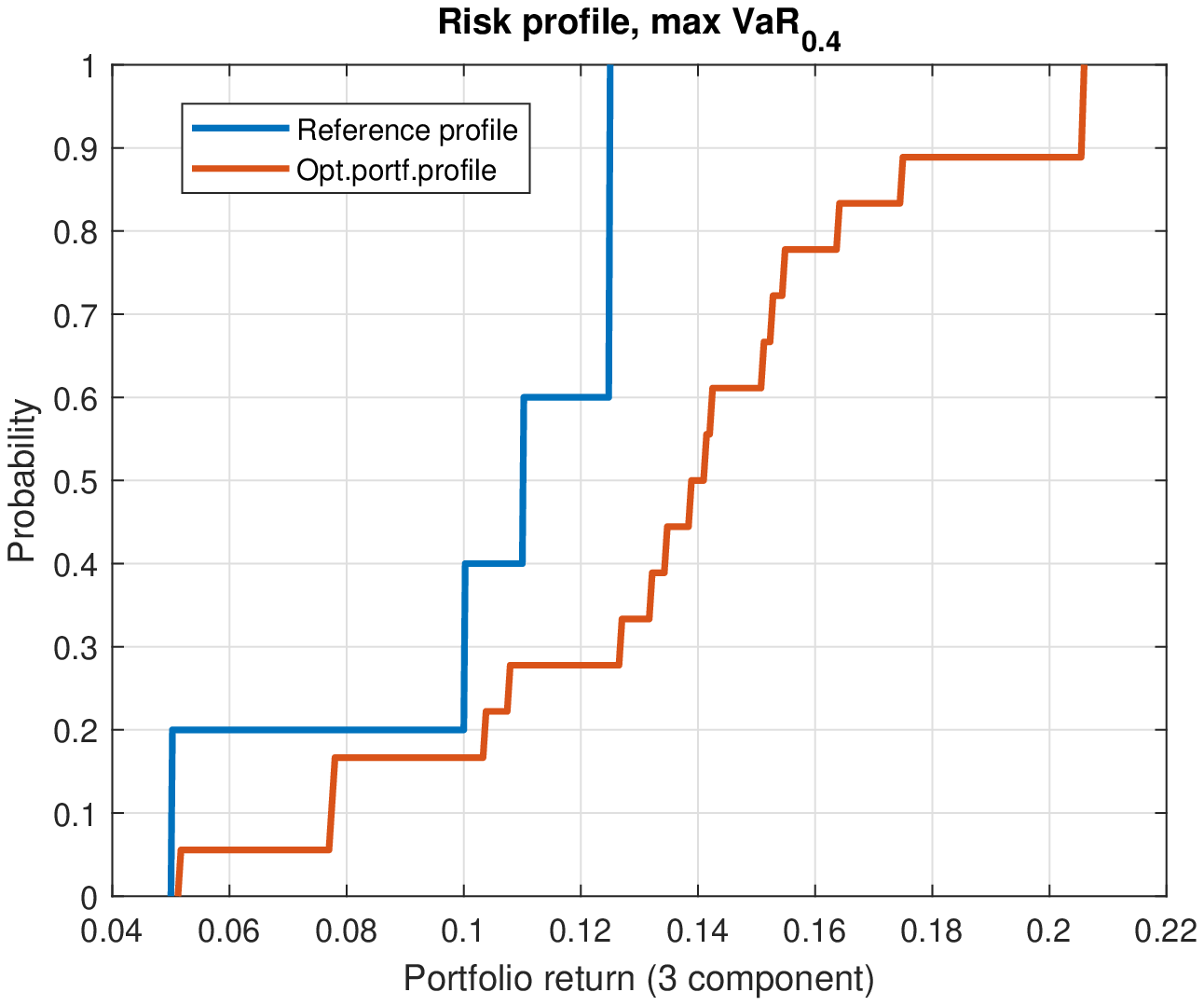}
	\includegraphics[width=0.335\textwidth]{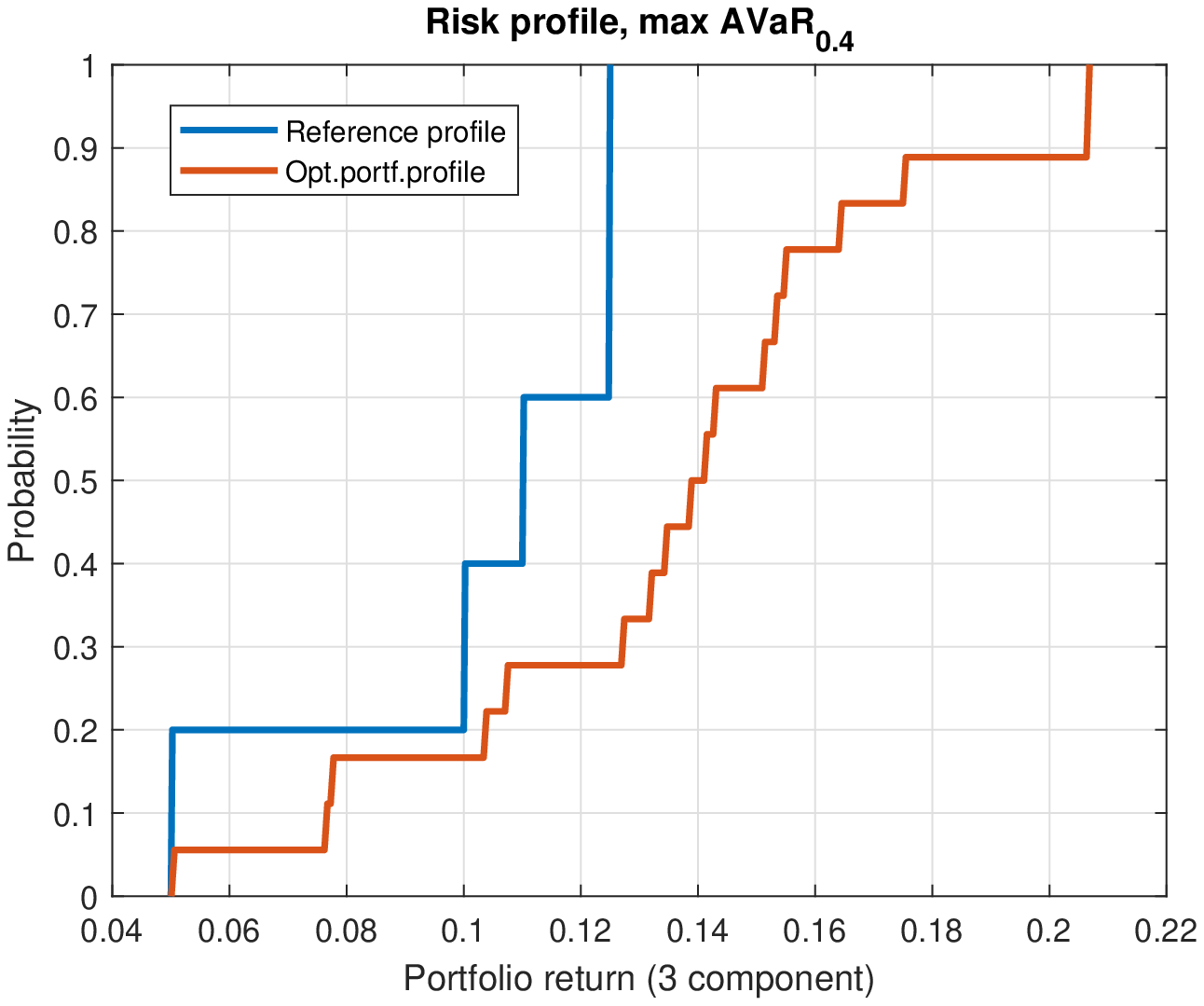}
\caption{Profiles of optimal 3 component portfolios: maximizing the tail returns under the risk-return lower bound. Analytical projective penalties.
	1) Optimal average return. 
	2) Optimal $\VaR_{0.4}$.  
	3) Optimal $\AVaR_{0.4}$. See Table~\ref{table_2022_12_03_ind_3comp_anal_lev0.4}.
}
\label{fig.2022.12.03_3comp_anal_lev0.4}       
\end{figure*}

The next two Tables~\ref{table_2022_12_03_ind_3comp} and~\ref{table_2022_12_03_ind}
(and the corresponding Figures~\ref{fig.2022.12.03_3comp} and~\ref{fig.2022.12.03})
show the effect of extension of a portfolio for account of new securities, from 3 to 10, for $\alpha= 70\,\%$.
The objective functions values in Table~\ref{table_2022_12_03_ind} are greater than the corresponding values in Table~\ref{table_2022_12_03_ind_3comp}. 
The corresponding Figures~\ref{fig.2022.12.03_3comp} and~\ref{fig.2022.12.03} indicate changes in the risk profiles of optimal portfolios due to this enlargement.
The increase of the objective functions happens also for account of huddling the risk profiles to the reference ones.
The pictures also show the influence of the different objective functions on the risk profiles of the optimal portfolios.
Finally, Table~\ref{table_2022_12_03_port} shows the structures of the optimal 10-component portfolios.

\begin{table*}[ht] \label{table_2022_12_03_ind_3comp}
	\begin{center}
	\caption{Optimal 3 component portfolios, analytical projection: 1)~Max mean. 2)~Max $\VaR$. 3)~Max $\AVaR$. See Figure~\ref{fig.2022.12.03_3comp}.}
	\begin{tabular}{rrrrrrr}
	\hline
	Portfolio No	&Mean	&$\VaR_{70\%}$	&$\AVaR_{70\%}$	&\;$x_4$	&$x_9$	&$x_{10}$	\\
	1	&0.1357	&0.1531 &0.1707 &0.1308	&0.0009	&0.8683	\\
	2	&0.1326	&0.1549	&0.1639	&0.0797	&0.0557	&0.8643 \\
	3	&0.1357 &0.1530	&0.1709	&0.1315	&0.0001	&0.8684 \\
	\hline
	\end{tabular}
	\end{center}
\end{table*}

\begin{figure*}[ht]
  \includegraphics[width=0.325\textwidth]{Fig_2022_12_03_3comp_AV_anal_proj.eps}
	\includegraphics[width=0.325\textwidth]{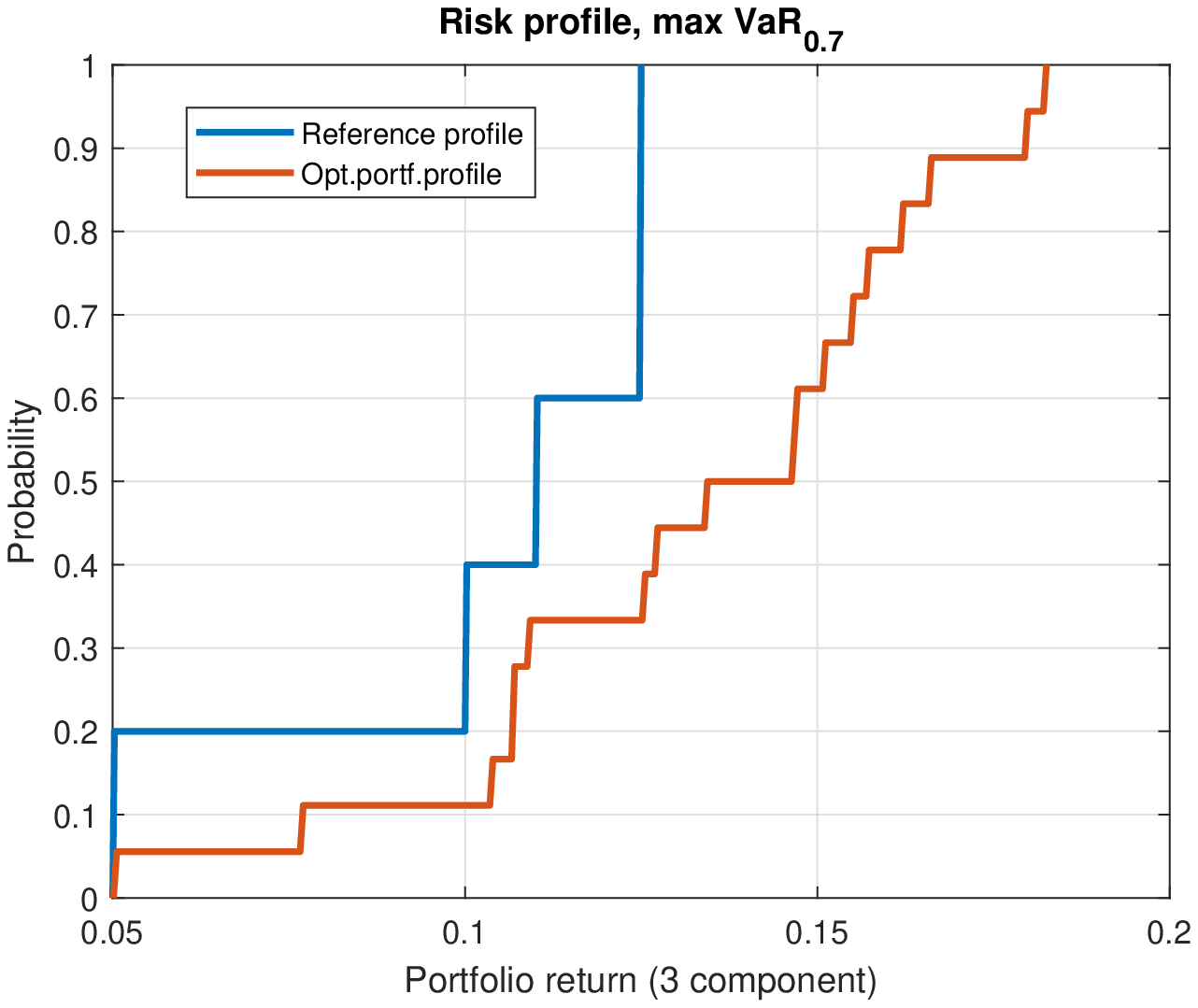}
	\includegraphics[width=0.335\textwidth]{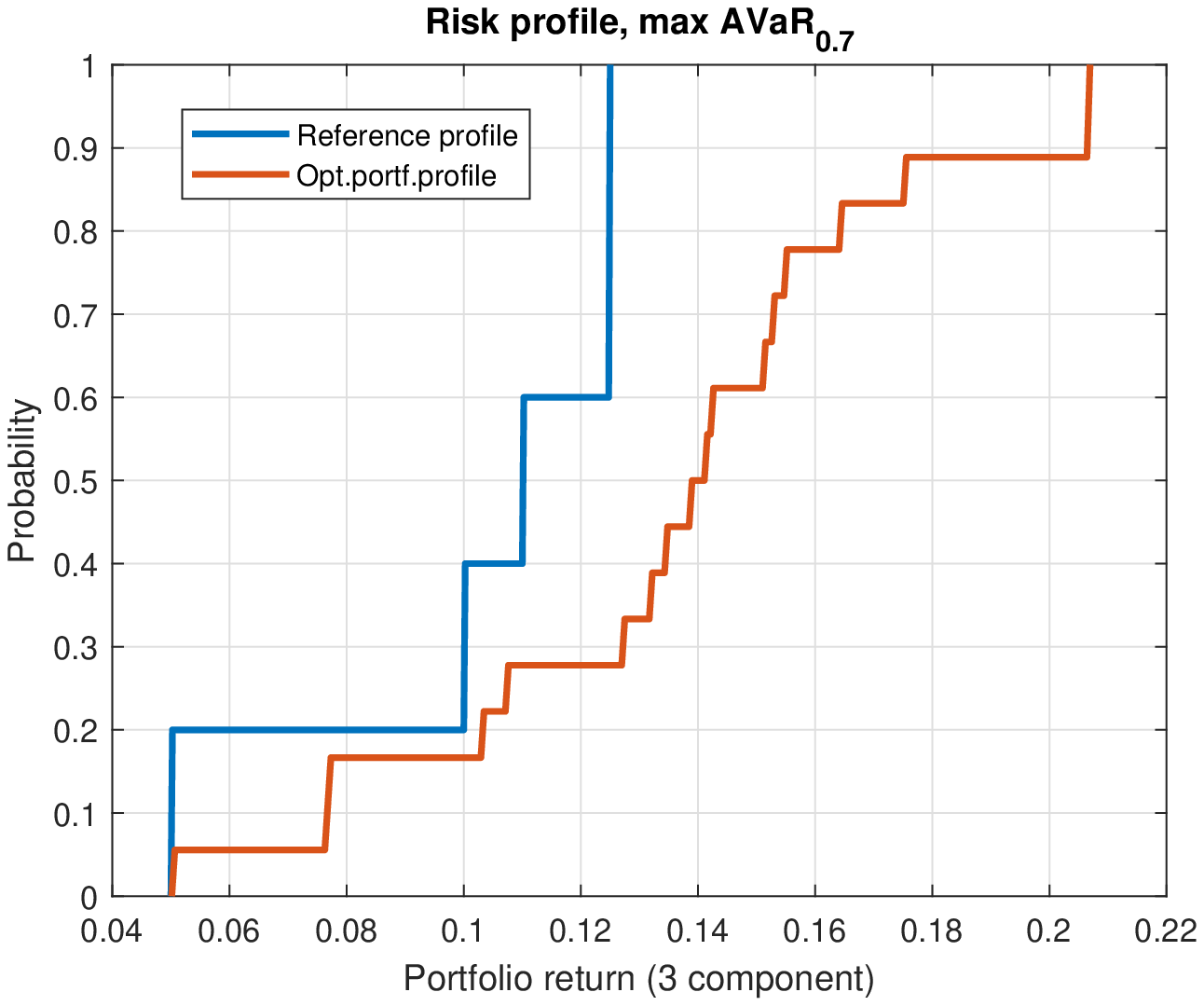}
\caption{Profiles of optimal 3 component portfolios: maximizing the 
tail returns under the risk-return lower bound.
Analytical projective penalties.
1) Optimal average return. 
2)  Optimal $\VaR_{70\%}$.  
3)  Optimal $\AVaR_{70\%}$. See Table~\ref{table_2022_12_03_ind_3comp}.
}
\label{fig.2022.12.03_3comp}       
\end{figure*}

\begin{table*}[ht]
\caption{Optimal 10 component portfolios, analytical projection: 1)~Max mean. 2)~Max $\VaR$. 3)~Max $\AVaR$. See Fig.~\ref{fig.2022.12.03}.}
\label{table_2022_12_03_ind}
\begin{tabular}{rrrrrrr}
\hline
Portf. No	&Mean	&$\VaR_{70\%}$	&$\AVaR_{70\%}$	&\;$x_4$	&$x_9$	&$x_{10}$	\\
1	&0.1380	&0.1643 &0.1781 &0.0126	&0.0022	&0.8615	\\
2	&0.1322	&0.1728	&0.1739	&0.0006	&0.0425	&0.8418 \\
3	&0.1367 &0.1629	&0.1901	&0.0126	&0.0002	&0.8490 \\
\hline
\end{tabular}
\end{table*}

\begin{figure*}[ht]
  \includegraphics[width=0.325\textwidth]{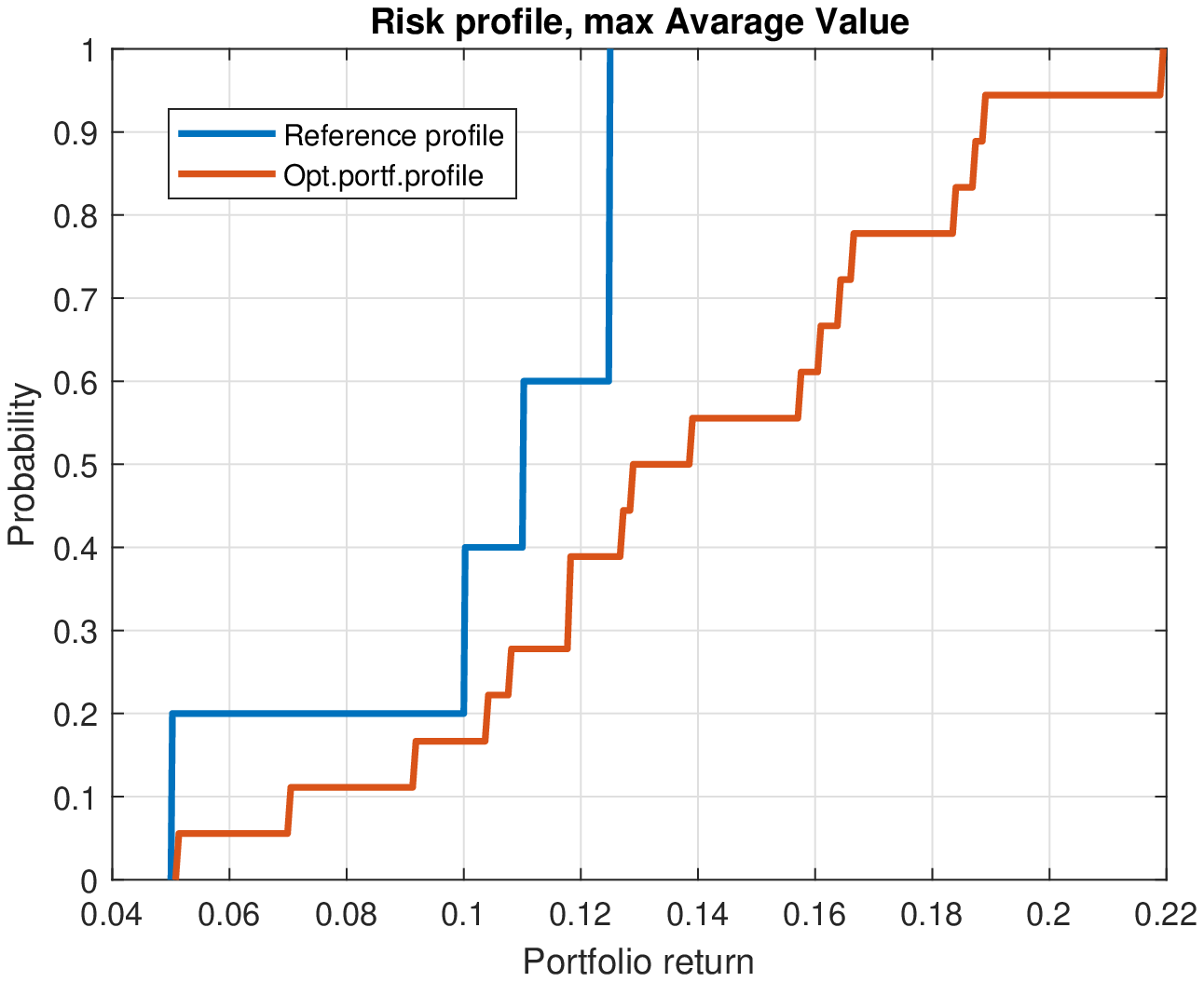}
	\includegraphics[width=0.325\textwidth]{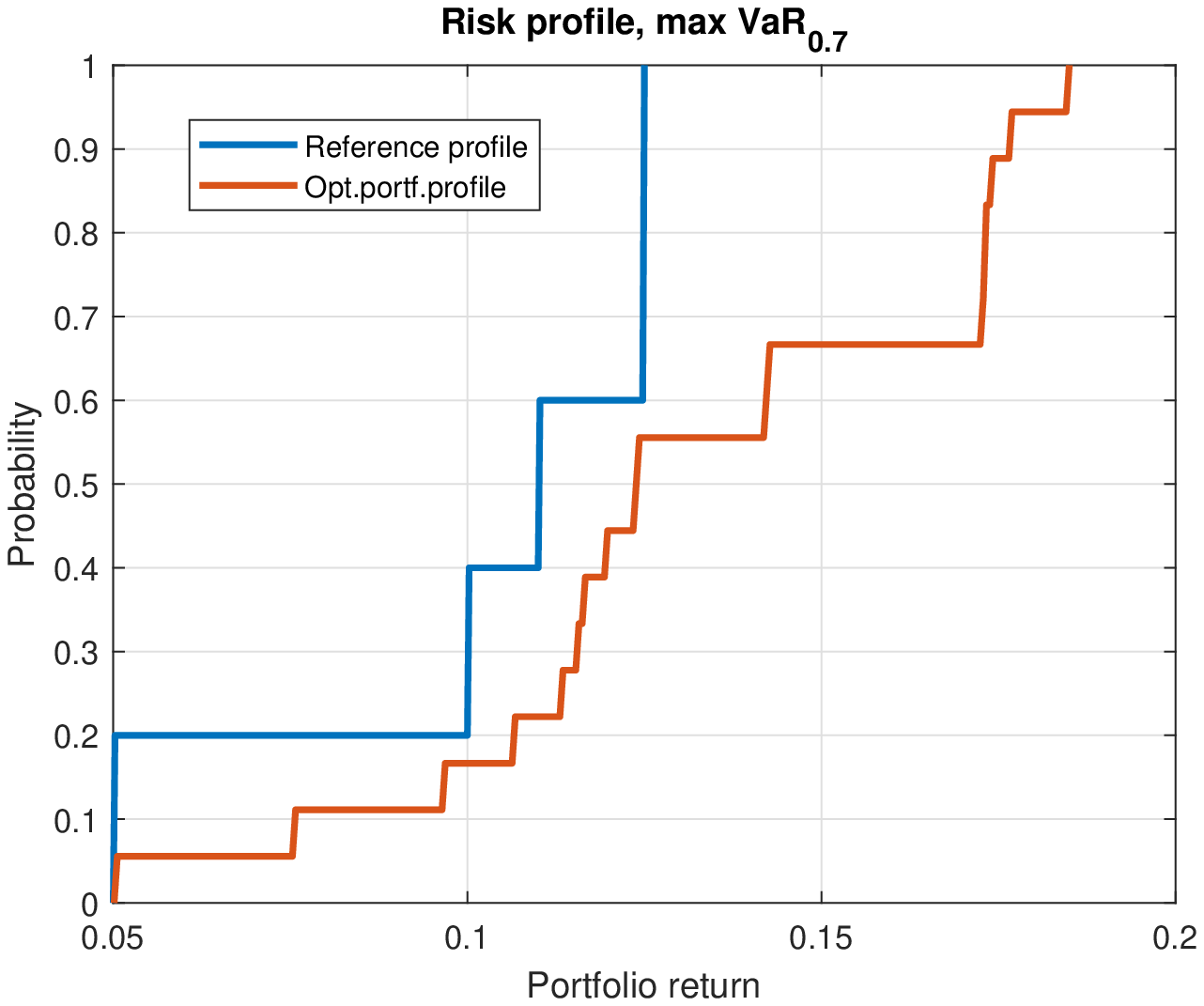}
	\includegraphics[width=0.335\textwidth]{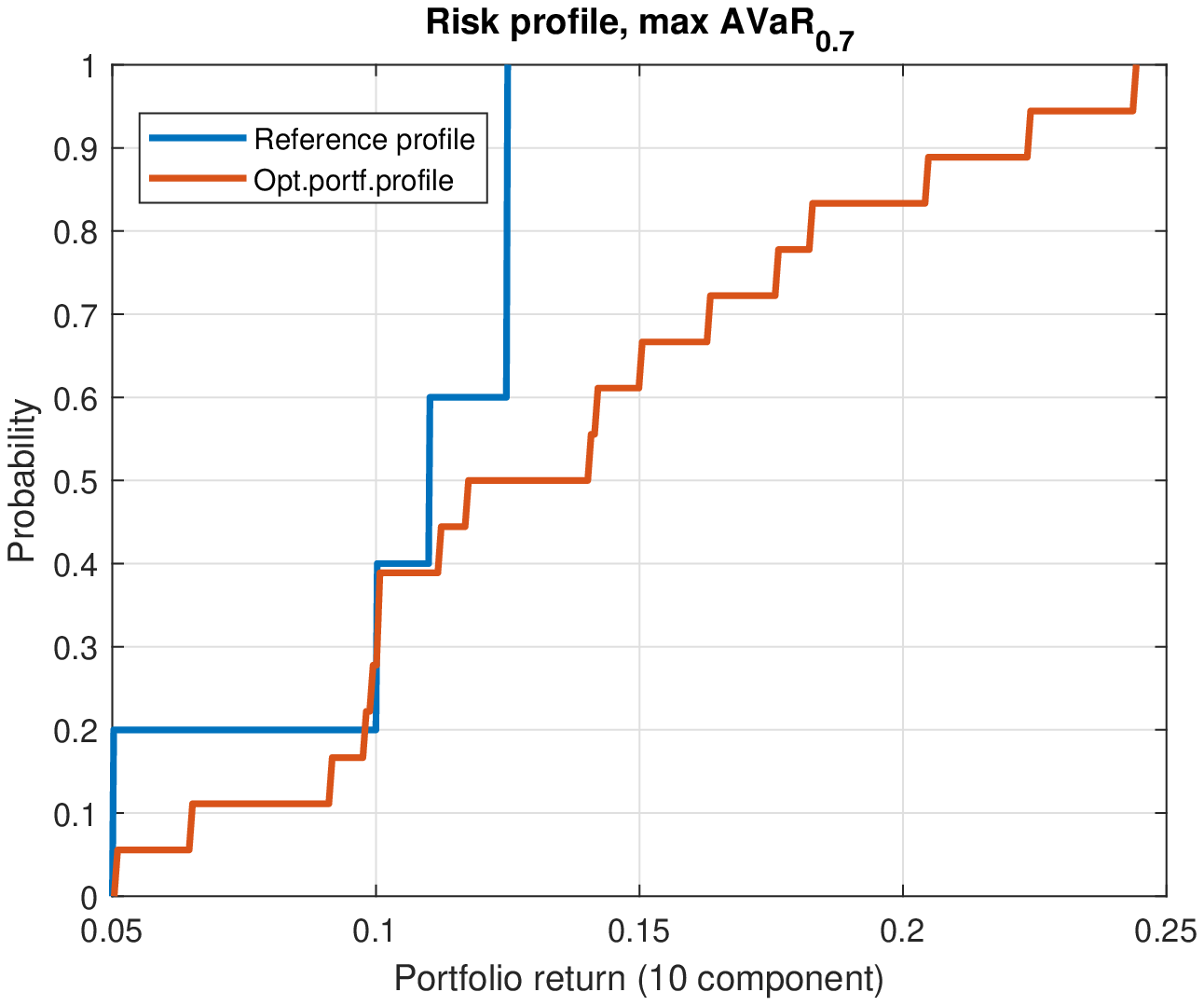}
\caption{Profiles of optimal 10 component portfolios: maximizing the 
tail returns under the risk-return lower bound. Analytical projective penalties.
1) Optimal average return. 
2) Optimal $\VaR_{70\%}$.  
3) Optimal $\AVaR_{70\%}$. See Tables~\ref{table_2022_12_03_ind} and~\ref{table_2022_12_03_port}.
}
\label{fig.2022.12.03}       
\end{figure*}
\clearpage

\begin{table*}[h]
\caption{The optimal 10 component portfolio}
\label{table_2022_12_03_port}
\begin{tabular}{rrrrrrrrrrr}
\hline
Portf.	&$x_1$	&$x_2$	&$x_3$	&$x_4$ &$x_5$	&$x_6$	&$x_{7}$ &$x_8$	&$x_9$ &$x_{10}$\\
Mean	      &.0013	&.0001  &.0059  &.0126 &.1011	&.0008	&.0031 &.0117 &.0022 &.8615\\
$\VaR_{0.7}$	&.0029	&.0103	&.0137	&.0006 &.0512	&.0128	&.0038 &.0202 &.0425 &.8418 \\
$\AVaR_{0.7}$&.0001	&.0001	&.0201	&.0126 &.0083	&.0005	&.0001 &.1090 &.0002 &.8490 \\
\hline
\end{tabular}
\end{table*}

\section{Conclusions}
The paper considers a specific method for optimization, which transforms a constraint optimization problem to an unconstrained, global optimization problem.
The paper illustrates the procedure for an optimization problem with uncountable many constraints.
More specifically, the paper considers financial portfolio optimization under 1-st order stochastic dominance constraints.
In the literature, similar portfolio optimization problems are mostly considered under 2\textsuperscript{nd} order stochastic dominance constraints, which constitutes a convex problem.
Few exceptions are \cite{Dentcheva_Ruszczynski_2004, Noyan_Rudolf_Ruszczynski_2006, Noyan_Ruszczynski_2008, Dentcheva_Ruszczynski_2013}.
The 1\textsuperscript{st} order constraints put lower bounds on the risk profile (CDF) of the optimized portfolio.
As objective functions, different aggregated indicators can serve, e.g., the expected value, the Value-at-Risk, or the average Value-at-Risk, etc.
In this setting, we put lower bounds on low returns and try to maximize higher returns.

Such constraints make the problem non-convex and hard for numerical treatment.
We propose the new exact penalty functions to handle the constraints and a new stochastic optimization (smoothing) techniques for solving penalty problems.
The approach is numerically and graphically illustrated on small test examples.
The advantage of the proposed approach to financial portfolio optimization consists in an additional visual control of the risk profile of the optimal portfolio.

\section{Acknowledgment}

    We would like to thank the editor of the journal and the referees for their commitment to assess and improve the paper.
    The authors gratefully acknowledge support by Volkswagenstiftung and DFG (Project-ID 4162\-28727 – SFB 1410).

\section{Data availability statement and conflict of interest}
	This study builds upon publicly available data collected in Table~\ref{table2}.
	The authors have no conflicts of interest to disclose. 

\bibliographystyle{abbrvnat}
\bibliography{LiteraturVladimir,LiteraturAlois}

\appendix

\begin{table} \label{table2}
\caption{Return data set from \cite[Table 1, page 13]{Markowitz_1959}, with artificial bond column}
\begin{tabular}{rrrrrrrrrrr}
\hline
Year	&Am.T.	&A.T.\&T.	&U.S.S.	&G.M.	&A.T.\&Sfe	&C.C.	&Bdn.	&Frstn.	&S.S. &Bond\\
1937	&-0.305	&-0.173	&-0.318	&-0.477	&-0.457	&-0.065	&-0.319	&-0.400	&-0.435 &0.125\\
1938	&0.513	&0.098	&0.285	&0.714	&0.107	&0.238	&0.076	&0.336	&0.238  &0.125\\
1939	&0.055	&0.200	&-0.047	&0.165	&-0.424	&-0.078	&0.381	&-0.093	&-0.295 &0.125\\
1940	&-0.126	&0.030	&0.104	&-0.043	&-0.189	&-0.077	&-0.051	&-0.090	&-0.036 &0.125\\
1941	&-0.280	&-0.183	&-0.171	&-0.277	&0.637	&-0.187	&0.087	&-0.400	&-0.240 &0.125\\
1942	&-0.003	&0.067	&-0.039	&0.476	&0.865	&0.156	&0.262	&1.113	&0.126  &0.125\\
1943	&0.428	&0.300	&0.149	&0.225	&0.313	&0.351	&0.341	&0.580	&0.639  &0.125\\
1944	&0.192	&0.103	&0.260	&0.290	&0.637	&0.233	&0.227	&0.473	&0.282  &0.125\\
1945	&0.446	&0.216	&0.419	&0.216	&0.373	&0.349	&0.352	&0.229	&0.578  &0.125\\
1946	&-0.088	&-0.046	&-0.078	&-0.272	&-0.037	&-0.209	&0.153	&-0.126	&0.289  &0.125\\
1947	&-0.127	&-0.071	&0.169	&0.144	&0.026	&0.355	&-0.099	&0.009	&0.184  &0.125\\
1948	&-0.015	&0.056	&-0.035	&0.107	&0.153	&-0.231	&0.038	&0.000	&0.114  &0.125\\
1949	&0.305	&0.038	&0.133	&0.321	&0.067	&0.246	&0.273	&0.223	&-0.222 &0.125\\
1950	&-0.096	&0.089	&0.732	&0.305	&0.579	&-0.248	&0.091	&0.650	&0.327  &0.125\\
1951	&0.016	&0.090	&0.021	&0.195	&0.040	&-0.064	&0.054	&-0.131	&0.333  &0.125\\
1952	&0.128	&0.083	&0.131	&0.390	&0.434	&0.079	&0.109	&0.175	&0.062  &0.125\\
1953	&-0.010	&0.035	&0.006	&-0.072	&-0.027	&0.067	&0.210	&-0.084	&-0.048 &0.125\\
1954	&0.154	&0.176	&0.908	&0.715	&0.469	&0.077	&0.112	&0.756	&0.185  &0.125\\
\hline
\end{tabular}
\end{table}

\end{document}